\documentclass{amsart}
\usepackage{amssymb}
\usepackage{eucal}
\usepackage{amsfonts}
\usepackage{epsfig}
\usepackage{color}
\usepackage[all]{xy}
\usepackage[pdftex]{hyperref}
\let\oldcite\cite                                  %   Comment this out
\newtheorem{thm}{Theorem}[section]
\newtheorem{cor}[thm]{Corollary}
\newtheorem{lem}[thm]{Lemma}
\newtheorem{prop}[thm]{Proposition}
\theoremstyle{definition}
\newtheorem{defn}[thm]{Definition}
\theoremstyle{remark}
\newtheorem{rem}[thm]{Remark}
\numberwithin{equation}{section} \theoremstyle{remark}
\newtheorem{ex}[thm]{Example}

\newcommand{\X}{\mathcal{X}}
\newcommand{\Y}{\mathcal{Y}}
\newcommand{\A}{\mathcal{A}}

\newcommand{\mfC}{\mathfrak{C}}
\newcommand{\mfB}{\mathfrak{B}}
\newcommand{\mfU}{\mathfrak{U}}
\newcommand{\mfS}{\mathfrak{S}}

\newcommand{\bbX}{\mathbb{X}}
\newcommand{\bbY}{\mathbb{Y}}
\newcommand{\bbE}{\mathbb{E}}
\newcommand{\bbU}{\mathbb{U}}

\newcommand{\Ym}{\Y_{mod}}

\newcommand{\liminv}{\operatorname{lim}}
\newcommand{\limdir}{\varinjlim}

\let\:=\colon

\newcommand{\al}{\alpha}
\newcommand{\be}{\beta}

\newcommand{\lra}{\longrightarrow}
\newcommand{\llra}[1]{\stackrel{#1}{\lra}}

\newcommand{\x}{\times}

\newcommand{\Top}{\mathsf{Top}}
\newcommand{\Para}{\mathsf{Para}}

\newcommand{\TopSt}{\mathfrak{TopSt}}
\newcommand{\HPTopSt}{\mathfrak{HPTopSt}}
\newcommand{\HoSt}{\mathfrak{HoSt}}
\newcommand{\PsSt}{\mathfrak{PsSt}}

\newcommand{\ParSt}{\mathfrak{ParSt}}
\newcommand{\HPParSt}{\mathfrak{HPParSt}}

\newcommand{\sfS}{\mathsf{S}}

\newcommand{\sfE}{\mathsf{E}}
\newcommand{\sfD}{\mathsf{D}}
\newcommand{\sfC}{\mathsf{C}}
\newcommand{\sfB}{\mathsf{B}}

\newcommand{\Ob}{\operatorname{Ob}}

\newcommand{\Hom}{\operatorname{Hom}}
\newcommand{\Map}{\operatorname{Map}}
\newcommand{\Roof}{\operatorname{Span}}
\newcommand{\id}{\operatorname{id}}
\newcommand{\pr}{\operatorname{pr}}
\def\smashedlongrightarrow{\setbox0=\hbox{$\longrightarrow$}\ht0=1pt\box0}
\def\risom{\buildrel\sim\over{\smashedlongrightarrow}}
\def\smashedst{\setbox0=\hbox{$\rightrightarrows$}\ht0=4pt\box0}
\newcommand{\sst}[1]{\stackrel{#1}{\smashedst}}

\newcommand{\hpara}{\text{hoparacompact}}
\newcommand{\hcw}{\text{hoCW}}
\newcommand{\sth}{F}

\begin{document}

\title{Homotopy types of topological stacks}

\author{Behrang Noohi}

% ----------------------------------------------------------------
\begin{abstract}
 We define the notion of {\em classifying space} of a topological stack and 
 show that every topological stack $\X$ has a classifying space $X$ 
 which is a topological space well-defined  up to weak homotopy equivalence. 
 Under a certain  paracompactness condition on
 $\X$, we show that $X$ is actually well-defined up to homotopy equivalence.
 These results are formulated in terms of functors from the category
 of topological stacks to the (weak) homotopy category of topological
 spaces.  We prove similar results for (small) diagrams of topological stacks.
\end{abstract}
% ----------------------------------------------------------------

\maketitle%

% ----------------------------------------------------------------
\section{Introduction}

The category of topological stacks accommodates various classes of
objects simultaneously: 1- orbifolds, 2- gerbes, 3- spaces with an
action of a topological group, 4- Artin stacks, 5- Lie groupoids, 6-
complexes-of-groups, 7- foliated manifolds. And, of course, the
category of topological spaces is a full subcategory of the category
of topological stacks.

In each of the above cases, tools from algebraic topology have been
adopted, to various extents, to study the objects in question. For
instance, (3) is the subject of equivariant algebraic topology.
Behrend has studied singular (co)homology of Lie groupoids in
\cite{Behrend}. Homotopy invariants and (co)homology theories for
orbifolds have been developed by Haefliger, Moerdijk,
Thurston,\dots. The case of complexes-of-groups has been studied
extensively by Bass, Bridson, Haefliger, Serre, Soul\'e,\dots

The notion of {\em classifying space} introduced in this paper 
(Definition \ref{D:classifying}) provides a unified
way to do algebraic topology on topological stacks.
By definition, a classifying space for a topological stack $\X$ is a topological space
$\Theta(\X)$ togeotehr with a morphism $\varphi \:
  \Theta(\X) \to \X$  which is a universal weak equivalence, in the sense that,
   for every map $T \to \X$ from a topological space
  $T$, the base extension $\varphi_T \: \Theta(\X)\times_{\X}T \to
  T$ of $\varphi$ is a weak equivalence of topological spaces. 
  
  The crucial feature of the definition of the classifying space is the existence
of the map $\varphi \:  \Theta(\X) \to \X$. This map provides a link between the
algebraic topology of $\X$ and that of the topological space $\Theta(\X)$ (say, by pull-back
and push-forward along $\varphi$). For an application of this, the reader can consult
\cite{BGNX} where this result is used to
develop an intersection theory and a theory of Thom classes on stacks;
for another application see \cite{EbGi}.

One of the main results of this paper is the following.

\begin{thm}{\label{T:1}}
  Every topological stack $\X$ admits an atlas  $\varphi \:
  \Theta(\X) \to \X$ which is a classifying space for $\X$.
\end{thm}

The classifying space $\Theta(\X)$ turns out to be unique up to weak homotopy equivalenvce.
In fact, we have  a functor $\Theta
\: \TopSt \to \Top_{w.e.}$ to the weak homotopy category
$\Top_{w.e.}$ of topological spaces which to a topological stack
$\X$ associates  its {\em weak homotopy type} $\Theta(\X)$; see
Theorem \ref{T:whtpytype}.

In the case where $\X$ has a groupoid presentation with certain
paracompactness properties ($\S$\ref{SS:homotopically}),  the above
functor can be lifted to a functor  $\Theta \: \TopSt \to
\Top_{h.e.}$ which associates to such a topological stack an actual {\em
homotopy type} (Theorem \ref{T:htpytype}). This applies to all
differentiable stacks and, more generally, to any stack that admits a
presentation by a metrizable groupoid; see Proposition
\ref{P:paracompact}. This result is useful when defining (co)homology
theories that are only homotopy invariant. (For example, certain
sheaf cohomology theories or certain \v{C}ech type theories are only
invariant under homotopy equivalences.)

The next main result in the paper is the generalization of Theorem
\ref{T:1}  to diagrams of topological stacks ($\S$\ref{S:Diagram}).
The following theorem is a corollary of Theorem
\ref{T:diagramadjoint}.

\begin{thm}
  Let $P \: \sfD \to \TopSt$ be a diagram of topological stacks
  indexed by a small category $\sfD$. Then, there is a diagram $Q \:
  \sfD \to  \Top$ of topological spaces, together with a
  transformation $\varphi \: P \Rightarrow Q $, such that for every
  $d \in \sfD$ the morphism $Q(d) \to P(d)$ is a universal weak
  equivalence. Furthermore, $Q$ is unique up to
  (objectwise) weak equivalence of diagrams.
\end{thm}

This theorem implies that every diagram of topological stacks has a
natural weak homotopy type as a diagram of topological spaces.
Furthermore, the transformation $\varphi$ relates the given diagram of
stacks with its weak homotopy type, thus allowing one to transport
homotopical information back and forth between the diagram and its
homotopy type.

The above theorem has various applications. For example, it implies
an equivariant version of Theorem \ref{T:1} for the (weak) action of
a discrete group. It also allows one to define homotopy types of
pairs (triples, and so on) of topological stacks. It is also
useful in studying Bredon type homotopy theories for topological
stacks.

We also consider the case where arrows in $\sfD$ are labeled by
properties of continuous maps, such as: subspace, open (or closed)
subspace, proper, finite, and so on. We show ($\S$\ref{SS:special})
that, under certain conditions, if morphisms in a diagram $P$ of
topological stacks have the properties assigned by the corresponding
labels, then we can arrange so that the morphisms in the homotopy
type $Q$ of $P$ also satisfy the same properties. For example, if we
take $\sfD$ to be $\{1 \to2\}$ and label the unique arrow of $\sfD$
by `closed subspace', this implies that the weak homotopy type of a
`closed pair' $(\X,\A)$ of topological stack can be chosen to be  a
`closed pair' $(X,A)$ of topological spaces. Furthermore, we have a
weak equivalence of pairs $\varphi \: (X,A) \to (\X,\A)$ relating
the pair $(\X,\A)$ to its weak homotopy type $(X,A)$.

The above results are valid for {\em arbitrary} topological stacks
$\X$. That is, all we require is for $\X$ to have a presentation by
a topological groupoid $\bbX=[X_1\sst{}X_0]$. Although this may
sound general enough for applications, there appears to be  a major
class of stacks which does not fall in this category: the mapping
stacks $\Map(\Y,\X)$ of topological stacks.

Nevertheless, in  \cite{Mapping} we prove that the mapping stacks
are not far from being topological stacks. Let us quote a result
from [ibid.].

\begin{thm}
  Let $\X$ and $\Y$ be topological stacks, and let $\Map(\Y,\X)$ be
  their mapping stack. If $\Y$ admits a presentation $[Y_1\sst{}
  Y_0]$ in which $Y_1$ and $Y_0$ are compact, then $\Map(\Y,\X)$ is
  a topological stack. If $Y_1$ and $Y_0$ are only locally compact,
  then $\Map(\Y,\X)$ is a paratopological stack.
\end{thm}

{\em Paratopological stacks} (Definition \ref{D:paratop}) form an
important 2-category of stacks which contains $\TopSt$ as a full sub
2-category. The advantage of the 2-category of paratopological
stacks over the 2-category $\TopSt$ of topological stacks is that it
is closed under arbitrary 2-limits (but we will not prove this
here). We show in $\S$\ref{S:Homotopical} that our machinery of
homotopy theory of topological stacks extends to the category of
paratopological stacks. In particular, all mapping stacks
$\Map(\Y,\X)$ have a well-defined (functorial) weak homotopy type,
as long as $\Y$ satisfies the locally compactness condition mentioned above.

We believe the category of paratopological stacks is a suitable category 
for doing homotopy theory in. Some other candidates are also discussed
in $\S$\ref{S:Homotopical}.

\medskip 

\noindent{\bf Acknowledgement.}
  I owe a big thank you to  Gustavo Granja for providing invaluable help in various stages of
  writing this paper.

% ----------------------------------------------------

\tableofcontents%

% ----------------------------------------------------
\section{Notation and conventions}{\label{S:Notation}}

Throughout the notes, $\Top$ stands for the category of all
topological spaces. The localization of $\Top$ with respect to the
class of weak equivalences is denoted by $\Top_{w.e.}$ (the {\em
weak homotopy category} of spaces). The localization of $\Top$ with
respect to the class of homotopy equivalences is denoted by
$\Top_{h.e.}$ (the {\em homotopy category} of spaces).

All stacks considered in the paper are over $\Top$.

We will denote groupoids by $\bbX=[s,t\: X_1 \sst{} X_0]$. For
convenience, we drop $s$ and $t$ from the notation. We usually
reserve the letters $s$ and $t$ for the source and target maps of
groupoids, unless it is clear from the context that they stand for
something else.

Our terminology differs slightly from that of \cite{Noohi}. A {\em
topological stack} means a stack $\X$ that is equivalent to the
quotient stack of a topological groupoid $\bbX=[X_1 \sst{} X_0]$; in
[ibid.] these are called {\em pretopological stacks}.

A morphism $f \: \X \to \Y$ of stacks is called an {\em epimorphism}
if it is an epimorphism in the sheaf theoretic sense (i.e., for
every topological space $W$, every object in $\Y(W)$ has a preimage
in $\X(W)$, possibly after passing to an open cover of $W$). In the
case where $\X$ and $\Y$ are topological spaces, this is equivalent
to saying that $f$ admits local sections.

For simplicity, we assume that all 2-categories have invertible
2-morphisms. The category obtained by identifying 2-isomorphic
morphisms in a 2-category $\mfC$ is denoted by $[\mfC]$. We usually
use Fraktur symbols for 2-categories and Sans Serif symbols for
1-categories.

% ---------------------------------------------------------------------
\section{Torsors for groupoids}{\label{S:Torsor}}

We quickly recall the definition of a torsor for a groupoid; see for
instance \cite{Noohi}, $\S$12.

\begin{defn}{\label{D:torsor}}
  Let $\bbX=[R \sst{} X]$ be a
  topological groupoid, and let $W$ be a topological space. By an
  $\bbX$-{\bf torsor} over $W$ we mean a map $p\: T \to W$ of topological
  spaces which admits local sections, together with a cartesian
  morphism of groupoids
     $$[T\x_W T \sst{} T] \to [R \sst{} X].$$
\end{defn}

By a {\em trivialization} of an $\bbX$-torsor $p \: T \to W$ we mean
an open covering $\{U_i\}$ of $W$, together with a collection of
sections $\sigma_i \: U_i \to T$. To give an  $\bbX$-torsor and a
trivialization for it is the same thing as giving a 1-cocycle on $W$
with values in $\bbX$, as defined below.

Given an open cover $\{U_i\}$ of $W$ of a topological space $W$, an
{\em $\bbX$-valued 1-cocycle} on $W$ relative the cover $\{U_i\}$
consists of a collection of continuous maps $a_i \: U_i \to X$, and
a collection of continuous maps $\gamma_{ij} \: U_i\cap U_j \to R$,
such that:
\begin{itemize}
  \item[$\mathbf{C1.}$] for every $i,j$,\quad
  $s\circ\gamma_{ij}=a_i|_{U_i\cap U_j}$ and $t\circ\gamma_{ij}=a_j|_{U_i\cap U_j}$;

  \item[$\mathbf{C2.}$] for every $i,j,k$,\quad
  $\gamma_{ij}\gamma_{jk}=\gamma_{ik}$ as maps from $U_i\cap U_j\cap
  U_k$ to $R$.
\end{itemize}

Equivalently, a 1-cocyle on $\{U_i\}$ is a groupoid morphism $c \: \bbU \to \bbX$,
where $\bbU:=[\coprod U_i\cap Uj \sst{} \coprod U_i]$ is the groupoid
associated to the covering $\{U_i\}$.

A morphism from a  1-cocycle $\big(\{U_i\},a_i,\gamma_{ij}\big)$ to
a 1-cocycle $\big(\{U'_k\},a'_k,\gamma'_{kl}\big)$ is a collection
of maps $\delta_{ik} \: U_i\cap U'_k \to R$ such that:
  \begin{itemize}
   \item[$\mathbf{M1.}$] for every $i,k$, \quad $s\circ\delta_{ik}=a_i|_{U_i\cap U'_k}$
   and $t\circ\delta_{ik}=a'_k|_{U_i\cap U'_k}$;

   \item[$\mathbf{M2.}$] for every $i,k,l$,
     $$\delta_{ik}\gamma'_{kl}=\delta_{il}  \: U_i \cap U'_k\cap U'_l \to R,$$
     $$  \gamma_{ij}\delta_{ik}=\delta_{jk} \: U_i \cap U_j\cap U'_k \to R.$$
  \end{itemize}

Equivalently, a morphism from the 1-cocycle $c\:\bbU \to \bbX$ to
the 1-cocycle
 $c'\:\bbU' \to \bbX$ is
 a morphism of groupoids $\bbU\coprod\bbU' \to \bbX$ whose restrictions
 to $\bbU$ and $\bbU'$ are equal to $c$ and $c'$, respectively.
 Here, $\bbU\coprod\bbU'$ is defined to be the groupoid
 associated to the covering $\{U_i,U_k\}_{i,k}$ (repetition allowed).

\begin{lem}{\label{L:cocycle}}
 Let $\bbX$ be a topological groupoid and $W$ a topological space.
 Then,  1-cocycles over $W$ and morphisms between them form a
 groupoid that is naturally equivalent to the groupoid of
 $\bbX$-torsors over $W$. This groupoid is also naturally equivalent
 to the groupoid $\Hom_{\mathbf{St}}(W,[X/R])$ of stack morphisms
 from $W$ to the quotient stack $[X/R]$.
\end{lem}

\begin{proof}
  The last statement can be found in (\cite{Noohi}, $\S$12).
  We only point out that the torsor $p \: T \to W$ associated to
  a morphism $\varphi \: W \to [X/R]$ is defined via the following
  2-cartesian diagram
   $$\xymatrix@=16pt@M=8pt{   T \ar[r] \ar[d]_p & X \ar[d] \\
                              W \ar[r]_{\varphi} & [X/R]} $$

  We explain how to associate an $\bbX$-torsor to a 1-cocycle
  $\big(\{U_i\},a_i,\gamma_{ij}\big)$. Set $T_i:=U_i\times_{a_i,X,s}R$.
  Define $T$ to be $\coprod T_i/_{\sim}$, where $\sim$ is the following
  equivalence relation: $(w_i,\al_i)\sim (w_j,\al_j)$, if
  $w_i=w_j=:w$ and $\al_i=\gamma_{ij}(w)\al_j$.

  The cartesian groupoid morphism $[T\times_W T \sst{} T] \to [R\sst{} X]$ is defined as follows.
  The map $T \to X$ is defined by $(w_i,\al_i) \mapsto t(\al_i)$.
  An element $\big((w_i,\al_i),(w_i,\be_i)\big)$ in $T\times_W T$ is mapped to
  $\al_i^{-1}\be_i  \in R$; this is easily seen to be well-defined (i.e.,
   independent of $i$).

   The rest of the proof is straightforward and is left to the
   reader.
\end{proof}

\begin{rem}
 Our definition of a 1-cocycle is different from that of (\cite{Haefliger},
 $\S$2)
 in that in {\em loc.\;cit.} the maps $a_i$ are not part of the data.
 There is, however, a forgetful map that associates to a 1-cocycle in our sense a cocycle
 in the sense of Haefliger.
\end{rem}

\section{Classifying space of a topological groupoid}{\label{S:Classifying}}

In this section, we discuss Haefliger's definition of the
classifying space of a topological groupoid.

% ----------------------------------------------------------------
\subsection{Construction of the classifying space of a
 topological groupoid}{\label{SS:classifying}}

We recall  from \cite{Haefliger} the definition of the (Haefliger-Milnor)
{\bf classifying space}
 $B\bbX$ and the {\bf universal bundle} $E\bbX$
 of a topological groupoid $\bbX=[R \sst{} X]$. Our main objective
 is to show that $\bbE \to \bbX$ is an $\bbX$-torsor, thereby giving
 rise to a morphism $\varphi \: B\bbX \to \X$, where $\X$ is the 
 quotient stack of $\bbX$. We will
 see in $\S$\ref{S:Nice} that $ B\bbX$ is a classifying space for $\X$ 
 in the sense of Definition \ref{D:classifying}

An element in $E\bbX$ is a sequence
$(t_0\al_0,t_1\al_1,\cdots,t_n\al_n,\cdots)$, where $\al_i \in R$
are such that $s(\al_i)$ are equal to each other,
 and $t_i \in [0,1]$ are such that all but finitely
many of them are zero and $\sum t_i=1$. As the notation suggests, we
set $(t_0\al_0,t_1\al_1,\cdots,t_n\al_n,\cdots)=(t'_0\al'_0,t'_1\al'_1,\cdots,t'_n\al'_n,\cdots)$
if $t_i=t'_i$ for all $i$ and $\al_i=\al'_i$ if $t_i\neq 0$.

Let $t_i \: E\bbX \to [0,1]$ denote the map
$(t_0\al_0,t_1\al_1,\cdots,t_n\al_n,\cdots) \mapsto t_i$, and
let $\al_i \: t_i^{-1}(0,1] \to R$ denote the
map $(t_0\al_0,t_1\al_1,\cdots,t_n\al_n,\cdots) \mapsto \al_i$. The
topology on $E\bbX$ is the weakest topology in which $t_i^{-1}(0,1]$ are all
open and $t_i$ and $\al_i$ are all continuous.

The classifying space $B\bbX$ is defined to be the quotient of
$E\bbX$ under the following equivalence relation. We say two
elements $(t_0\al_0,t_1\al_1,\cdots,t_n\al_n,\cdots)$ and
$(t'_0\al'_0,t'_1\al'_1,\cdots,t'_n\al'_n,\cdots)$ of $E\bbX$ are
equivalent, if $t_i=t'_i$  for all $i$, and if there is an element
$\gamma \in R$ such that $\al_i=\gamma\al'_i$. (So, in particular,
$t(\al_i)=t(\al'_i)$  for all $i$.) Let $p \: E\bbX \to B\bbX$ be
the projection map.

The projections $t_i \: E\bbX \to [0,1]$ are compatible with this
equivalence relation, so they induce continuous maps $u_i \: B\bbX
\to [0,1]$ such that $u_i\circ p=t_i$. Let $U_i=u_i^{-1}(0,1]$.

\begin{lem}{\label{L:torsor}}
  The projection map $p \: E\bbX \to B\bbX$ can be naturally made into a
  $\bbX$-torsor.
\end{lem}

\begin{proof}
   First we show that $p$ admits local sections. Consider
   the open cover $\{U_i\}$ of $B\bbX$ defined above.
   We define a section $U_i \to E\bbX$ for $p$ by sending
   the equivalence class of
   $(t_0\al_0,t_1\al_1,\cdots,t_n\al_n,\cdots)$
   to $(t_0\al_i^{-1}\al_0,t_1\al_i^{-1}\al_1,\cdots,t_i\al_i^{-1}\al_n,\cdots)$.

   Let us now define a cartesian groupoid morphism
      $$F \: [E\bbX\times_{B\bbX}E\bbX \sst{}E\bbX] \to [R\sst{} X].$$
   The effect on the object space is given by the map $f \: E\bbX \to X$
   which sends $(t_0\al_0,t_1\al_1,\cdots,t_n\al_n,\cdots)$
   to $s(\al_i)$; this is independent of $i$ (by definition).
   An element in $E\bbX\times_{B\bbX}E\bbX$ is represented
   by a pair
   $$\big((t_0\gamma\al_0,t_1\gamma\al_1,\cdots,t_n\gamma\al_n,\cdots),
    (t_0\al_0,t_1\al_1,\cdots,t_n\al_n,\cdots)\big),$$
   for a unique $\gamma \in R$. We send this element to $\gamma \in R$.
   This is easily verified to be a cartesian morphism of groupoids.
\end{proof}

It follows from Lemma \ref{L:cocycle} that we have a morphism
natural $\varphi \: B\bbX \to \X$ and the the $\bbX$-torsor $p \:
E\bbX \to B\bbX$ fits in a 2-cartesian diagram
   $$\xymatrix@=16pt@M=8pt{   E\bbX \ar[r]^{f} \ar[d]_p & X \ar[d] \\
                              B\bbX \ar[r]_{\varphi} & \X} $$

% --------------------------------------------------------------------
\subsection{Comparison with the simplicial construction}{\label{SS:simplicial}}

There is another way of associating a classifying space to a
groupoid $\bbX=[R\sst{}X]$, namely, by taking  the geometric
realization of (the simplicial space associated to) it. In this
subsection, we compare this construction with Haefliger's and
explain why we prefer Haefliger's construction.

First, let us recall the construction of the geometric realization
of $\bbX$. Consider the simplicial space $N\bbX$ with
 $$(N\bbX)_0=X, \ \ \text{and} \ \ (N\bbX)_n=
   \underbrace{R\times_{X}\times\cdots\times_{X}\times R}_{n-\text{fold}},
                                            \ n\geq 1.$$
The geometric realization  of $\bbX$ is, by definition, the
geometric realization  of $N\bbX$. We will denote it by $|\bbX|$.

Alternatively, $|\bbX|$ can be obtained as a quotient space of
$E\bbX$ by declaring that ``it is allowed to take common factors in
$E\bbX$.'' This means that, if an element $\alpha \in R$ appears
several times in the sequence
$\mathfrak{s}=(t_0\al_0,t_1\al_1,\cdots,t_n\al_n,\cdots) \in E\bbX$,
say at indices $i_1,\cdots,i_k$, then we regard $\mathfrak{s}$ as
equivalent to any sequence  $\mathfrak{s}'\in E\bbX$ which is
obtained from $\mathfrak{s}$ by altering the coefficients $t_{i_1},
\cdots,t_{i_k}$ in a way that $t_{i_1}+\cdots+t_{i_k}$ remains
fixed. (Roughly speaking, we are collapsing the subsequence
$(t_{i_1}\alpha, \cdots,t_{i_k}\alpha)$ of $\mathfrak{s}$ to a
single element $(t_{i_1}+\cdots+t_{i_k})\alpha$.)

It can be shown that there is  a universal bundle $|\bbE|$ over
$|\bbX|$ that is {\em almost} an $\bbX$-torsor. We explain how it is
defined.

Consider the topological groupoid $\bbE:=[R\times_{s,X,s}R\sst{}
R]$. There is a groupoid morphism $p \: \bbE \to \bbX$ induced by
the target map $t \: R \to X$. This induces a map $|p| \: |\bbE| \to
|\bbX|$ on the geometric realizations. This will be the structure
map of our (almost) torsor. Let us explain how the cartesian
morphism
  $$[|\bbE|\times_{|\bbX|}|\bbE|\sst{} |\bbE|] \to [R\sst{}X]$$
is constructed. Viewing $X$ and $R$ as trivial groupoids
$[X\sst{}X]$ and $[R\sst{}R]$, we have the following strictly
cartesian diagram in the category of topological groupoids:
   $$\xymatrix{ \bbE\times_{\bbX}\bbE \ar[r]^{\lambda} \ar[d]_{\pr_1 \
   \text{or} \ \pr_2}
                     & R  \ar[d]^{s \ \text{or} \ t}\\
        \bbE \ar[r]_{\sigma} & X    }$$
In this diagram, the map $\sigma \: [R\times_{s,X,s}R\sst{} R] \to
[X\sst{}X]$ is the one induced by the source map $s \: R \to X$. The
fiber product $\bbE\times_{\bbX}\bbE$ is the strict fiber product of
groupoids, and the maps $\bbE \to \bbX$ appearing in this fiber
product are both $p$. The morphism $\lambda$ in the diagram is
defined as follows. An object in the groupoid
$\bbE\times_{\bbX}\bbE$ is a pair of arrows $(\gamma_1,\gamma_2)$
with the same target. Under $\lambda$, this will get sent to
$\gamma_1\gamma_2^{-1} \in R$. The effect of $\lambda$ on arrows is
now uniquely determined.

The above diagram is indeed a (cartesian) morphism of groupoid
objects in the category of topological groupoids. After passing to
geometric realizations at all four corners, and noting that taking
geometric realizations commutes with fiber products, the above
diagram gives rise to the desired cartesian morphism of topological
groupoids $\Psi \: [|\bbE|\times_{|\bbX|}|\bbE|\sst{} |\bbE|] \to
[R\sst{}X]$.

This almost proves that $|p| \: |\bbE| \to |\bbX|$ is an
$\bbX$-torsor. The only thing that is left to  check is the
existence of local sections. This, however, may not be true in
general, unless the source and target maps of the original groupoid
$\bbX$, and also its identity section, are nicely behaved (locally).
This prevents $p \: |\bbE| \to |\bbX|$ from being an $\bbX$-torsor.
As a consequence, we do {\em not}
get a morphism $|\bbX| \to [X/R]$. %\FIX{be more precise about this}
This explains why we opted for $B\bbX$ rather than $|\bbX|$ as a
model for the classifying space of $\bbX$.

\begin{rem}{\label{R:simp}}
The above discussion can be summarized by saying that there are
quotient maps $q' \: E\bbX \to |\bbE|$ and   $q\: B\bbX \to |\bbX|$
inducing a commutative diagram
     $$\xymatrix@=16pt@M=8pt{
      [E\bbX\times_{B\bbX}E\bbX \sst{} E\bbX] \ar[r]^{Q} \ar[rd]_{\Phi} &
              [|\bbE|\times_{|\bbX|}|\bbE|\sst{} |\bbE|]  \ar[d]^{\Psi} \\
                                                       & [R\sst{} X]} $$
of cartesian groupoid morphisms. The  morphism $\Phi$ {\em does}
make $E\bbX \to B\bbX$ into an $\bbX$-torsor. In contrast, the
morphism $\Psi$ {\em does not} always make $|\bbE| \to |\bbX|$ into
an $\bbX$-torsor. Therefore, the dotted arrow in the following
diagram of the corresponding quotient stacks
     $$\xymatrix@=16pt@M=8pt{  B\bbX \ar[r]^{q} \ar[rd]_{\varphi}^{\cong} &
                                      |\bbX|  \ar@{..>}[d] \\
                                             & [X/R]} $$
may not always be filled.
\end{rem}

% --------------------------------------------------------------------
\subsection{The case of a group action}{\label{SS:group}}

Let $G$ be a topological group acting continuously on a topological
space $X$ (on the right). Recall that  $\X=[X/G]$ is the quotient
stack of the topological groupoid $\bbX=[X\times G \sst{} X]$. In
this case, $B\bbX$ is equal to the Borel construction, that is
$B\bbX=X\times_G  EG$. The torsor $E\bbX$ constructed in
$\S$\ref{S:Classifying} is equal to $X\times EG$, and $p \: E\bbX
\to B\bbX$ is Milnor's universal $G$-bundle. The cartesian square
appearing after the proof of Lemma \ref{L:torsor} now takes the
following form
 $$\xymatrix@=16pt@M=8pt{  X\times EG  \ar[r]^f \ar[d]_p & X \ar[d] \\
                           X\times_G  EG  \ar[r]_{\varphi} & [X/G]} $$

% ------------------------------------------------
\section{Shrinkable  morphisms}{\label{S:Shrinkable}}

We begin with an important definition.

\begin{defn}{\label{D:retraction}}
  We say that a continuous map  $f \: X \to Y$ of topological spaces is
  {\bf shrinkable} (\oldcite{Dold}, $\S$1.5), if it admits a section $s \: Y \to X$
  such that there is a fiberwise strong deformation retraction of $X$ onto
  $s(Y)$. We say that $f$ is   {\bf locally shrinkable},
  if there is an open cover $\{U_i\}$ of $Y$ such
  that $f|_{U_i} \: f^{-1}(U_i) \to U_i$ is shrinkable for all $i$.
  We say that $f$ is {\bf parashrinkable}, if for every map $T \to Y$
  from a paracompact topological space $T$, the base extension
  $f_T \: T\times_{Y}X \to T$ is shrinkable. If this condition is
  only satisfied for $T$ a CW complex, we say that $f$ is {\bf
  pseudoshrinkable}.
  We say that  $f$  is a {\bf universal weak equivalence}, if
  for every map $T \to Y$ from a topological space $Y$, the base
  extension $f_T \: T\times_Y X \to T$ of $f$ is a weak equivalence.
\end{defn}

\begin{defn}{\label{D:stackshrinkable}}
  We say that a representable morphism $f \: \X \to \Y$ of topological
  stacks is locally shrinkable (respectively, parashrinkable,
  pseudoshrinkable, a universal weak equivalence) if for every map $T \to
  \Y$ from a topological space $Y$, the base
  extension $f_T \: T\times_Y \X \to T$ of $f$ is so.
\end{defn}

\begin{rem}
  The above notions do not distinguish 2-isomorphic morphisms of stacks, so they pass to $[\TopSt]$.
\end{rem}

The following lemma clarifies the relation between the above
notions.

\begin{lem}{\label{L:shrinkable}}
   The properties introduced in Definition \ref{D:stackshrinkable}
   are related in the following way:
     $$\xymatrix@=0pt@M=3pt{   &&\text{triv. Hurewicz fib.} &\Rightarrow & \text{triv. Serre fib.}&& \\
                        &\rotatebox{-160}{$\Rightarrow$} && &  & \rotatebox{-30}{$\Rightarrow$} &  \\
 \text{shrinkable} & \Rightarrow & \text{locally shrinkable} &
       \Rightarrow & \text{parashrinkable} & \Rightarrow &
                                     \text{pseudoshrinkable} \\
      &&\rotatebox{-90}{$\Rightarrow$}&&&&
                               \rotatebox{-90}{$\Rightarrow$} \\
                      &&\text{epimorphism} &&&&
                                   \text{universal weak eq.} }$$
\end{lem}

\begin{proof}
  All the implications are obvious except for the one in the middle
  and the one on the top-left. The one on the top-left follows from
  \cite{Dold}, Corollary 6.2. To prove the middle implication, we
  have to show that every  locally shrinkable $f \: X \to Y$ with
  $Y$ paracompact is shrinkable.  Dold (\cite{Dold}, $\S$2.1)
  proves that if $f \: X \to Y$ is a continuous map which becomes
  shrinkable after passing to a numerable (\cite{Dold}, $\S$2.1)
  cover $\{U_i\}_{i\in I}$ of $Y$, then $f$ is shrinkable. Since in
  our case $Y$ is paracompact,  every open cover of $Y$ admits a
  numerable refinement.  So $f$ is shrinkable by Dold's result.
\end{proof}

\begin{lem}{\label{L:lift}}
  Let $f \: \X \to \Y$ be a parashrinkable
  (respectively, pseudoshrinkable) morphism of
  topological stacks. Let $B$ be a paracompact topological space
  (respectively, a CW complex).
   Then, for every morphism $g \: B \to \Y$, the space of lifts
   $g$ to $\X$ is non-empty and contractible. In particular,
  every morphism $g \: B \to \Y$ has a lift $\tilde{g} \: B \to \X$
  and such a lift is unique up to homotopy.
\end{lem}

\begin{proof}
  Recall that a lift  of $g$ to $\X$ means a map $\tilde{g} \: B \to
  \X$ together with a 2-isomorphism $\varepsilon \:
  f\circ\tilde{g}\Rightarrow g$.

  The space of lifts of $g$ is homeomorphic to the space of sections
  of the map shrinkable map $f_B \: B\times_{\X}\Y \to B$, hence is
  contractible.
\end{proof}

The converse of Lemma \ref{L:lift} is also true and can be used as
an alternative way of defining parashrinkable (respectively,
pseudoshrinkable) morphisms.

\begin{lem}{\label{L:epibasechange0}}
   Let $f \: \X \to \Y$ and $g \: \Y' \to \Y$ be  morphisms of
  topological stacks. Let $f' \:\X' \to \Y'$ be the base extension
  of $f$ along $g$. If $f$ is  locally shrinkable (respectively,
  parashrinkable, pseudoshrinkable, a universal weak equivalence),
  then so is $f'$. If $g$ is an epimorphism, and $f'$ is locally
  shrinkable, then $f$ is also locally shrinkable.
\end{lem}

\begin{proof}
Obvious.
\end{proof}

\begin{lem}{\label{L:retractionlift}}
Consider the
 2-commutative square
  $$\xymatrix@=8pt@M=6pt{
         A   \ar[rr]^f\ar[dd]_{i}   & \ar@{=>}[dl]^{\varphi}&
              \X  \ar[dd]^{p}  \\ & & \\
         B    \ar[rr]_g
        & &  \Y    }$$
in which  $p \: \X \to \Y$ is a parashrinkable (respectively, a
pseaudoshrinkable) morphism of topological stacks and $i \:A
\hookrightarrow B$ is a closed Hurewicz cofibration of paracompact
topological spaces (respectively, CW complexes). Then, one can find
$h$ and $\alpha$ such that in the diagram
 $$\xymatrix@=8pt@M=6pt{
         A   \ar[rr]^f\ar[dd]_{i}   & &
              \X  \ar[dd]^{p}  \\ \ar@{:>}[ur]_{\al} & &  \\
         B \ar@{..>}[uurr] |-{h}   \ar[rr]_g
        & &  \Y    }$$
 the upper triangle is 2-commutative and the lower triangle commutes
 up to a homotopy that leaves $A$ fixed (i.e., there is a homotopy
 from $g$ to $p\circ h$ which after precomposing with $i$ becomes
 2-isomorphic to the constant homotopy).
\end{lem}

\begin{proof}
  We can replace $p$ by its base extension $p_B \: B\times_{\X}\Y
  \to B$. So, we are reduced to the following situation: we have
  shrinkable map $p \: X \to B$ of topological spaces, a subspace $A
  \subset B$ which is a Hurewicz cofibration, and a section $s \: A
  \to X$. We want to extend $s$ to a map $\sigma \: B \to X$ such
  that $p\circ\sigma$ is homotopic to the identity map $\id_B \: B
  \to B$ via a homotopy that fixes $A$ pointwise.

  Since $p$ is shrinkable, it has a section $S \: B \to X$.
  Furthermore, $S|_A$ and $s$ are fiberwise homotopic via a homotopy
  $H \: A\times[0,1] \to X$. (Fiberwise means that $p\circ H \:
  A\times[0,1] \to A$ is equal to the first projection map.)  Set
        $$C=(A\times[0,1])\cup_{A\times\{0\}}B.$$
  The maps $H$ and $S$ glue to
  give a map $k \: C \to X$. The composition $p\circ k \: C \to B$
  is the collapse map that fixes $B$ pointwise and collapses
  $A\times[0,1]$ onto $A$ via projection.

  Since $A \subset B$ is a closed Hurewicz cofibration, the
  inclusion $C \subset B\times[0,1]$ admits a retraction $r\:
  B\times[0,1] \to C$ (\cite{Whitehead}, $\S$I.5.2). Let $r_1 \: B
  \to C$ be the restriction of $r$ to $B\times\{1\}$. Set $\sigma:=
  k \circ r_1$. It is easy to see
  that $\sigma$  has the desired property.
\end{proof}

\begin{rem}
  If in the above lemma we switch the roles of the upper and the
  lower triangles, namely, if we require that the upper triangle is
  homotopy commutative and the lower one is 2-commutative, then the
  lemma is true without the cofibration assumptions on  $i$.
\end{rem}

\begin{cor}{\label{C:retractionwe}}
  Let $f \: \X \to \Y$ be a pseudoshrinkable morphism of Serre
  topological stacks (\cite{Noohi}, $\S$17; also see Definition
  \ref{D:Serre}). Then, $f$ is a  weak equivalence, i.e., it induces
  isomorphisms on all homotopy
  groups (as defined in \cite{Noohi}, $\S$17).
\end{cor}

\begin{proof}
 To prove the surjectivity of $\pi_n(\X,x) \to \pi_n(\Y,f(x))$,
 apply Lemma \ref{L:retractionlift} to the case where $B$ is $S^n$ and
 $A$ is its base point.
 The injectivity follows by considering $B=\mathbb{D}^{n+1}$ and
 $A=\partial\mathbb{D}^{n+1}$.
\end{proof}

\begin{lem}{\label{L:prod}}
  Consider a family $f_i \: \X_i \to \Y$, $i \in I$, of
  representable morphisms of topological stacks, and let $f \:
  \prod_{\Y} \X_i \to \Y$ be their fiber product. (Note that
  $f$ is well-defined up to a 2-isomorphism.) If all $f_i$ are
  parashrinkable (respectively, pseudoshrinkable), then so is $f$.
  If all $f_i$ are  locally shrinkable (respectively, universal weak
  equivalence), then so is  $f$, provided $I$ is finite.
\end{lem}

\begin{proof}
  In the parashrinkable case, it is enough to assume that $\Y=Y$ is
  a paracompact topological space. The result now follows from the
  fact that an arbitrary fiber product of shrinkable morphisms $X_i
  \to Y$, $i \in I$, of topological spaces is shrinkable.  The case
  of pseudoshrinkable morphisms is proved analogously.

  The case of locally shrinkable maps is also similar, except that
  for each $i \in I$ we may have a different open cover of $Y$ which
  makes $f_i$ shrinkable. Since $I$ is finite, choosing a common
  refinement will  make $f$ shrinkable.

  The case of universal weak equivalence maps is easily proved by
  induction (using the fact that the composition of two universal
  weak equivalences is again a universal weak equivalence).
\end{proof}

% ----------------------------------------------------------------
 \section{Classifying space of a topological
 stack}{\label{S:Nice}}

In this section we prove our first main result (Theorem \ref{T:nice}). We begin with an important
proposition.

\begin{prop}{\label{P:main}}
  Let $\X$ be a topological stack, and let $\bbX=[R \sst{} X]$
  be a presentation for it. Then, there is a natural map $\varphi\: B\bbX
  \to \X$  which fits in the following
   2-cartesian diagram
   $$\xymatrix@=16pt@M=8pt{   E\bbX \ar[r]^{f} \ar[d]_p & X \ar[d] \\
                              B\bbX \ar[r]_{\varphi} & \X} $$
  Here, $f$ is the map defined in  the proof of Lemma
  \ref{L:torsor}. Furthermore, $f$ is shrinkable. In particular,
  $\varphi$  is  locally shrinkable.
\end{prop}

\begin{proof}
  The $\bbX$-torsor $p\: E\bbX \to B\bbX$ defined in Lemma \ref{L:torsor}
  furnishes the map $\varphi$ and the 2-cartesian diagram; see Lemma \ref{L:cocycle}.

 Let us show that $f$ is shrinkable.
   Define the section $\sigma \: X \to E\bbX$ by
  $$x \mapsto (1\id_x,0\id_x,\cdots,0\id_x,\cdots).$$
  This identifies $X$ with a closed subspace of $E\bbX$.
  We define the desired strong deformation retraction
   $\Psi \: [0,2]\times E\bbX \to E\bbX $ by juxtaposing
   the maps $\Psi_1 \: [0,1]\times E\bbX \to E\bbX $
   and $\Psi_2 \: [1,2]\times E\bbX \to E\bbX $ which are defined as follows:
{\small
$$\Psi_1 \: \big(t,(t_0\al_0,t_1\al_1,\cdots,t_n\al_n,\cdots)\big) \mapsto
      \big((t_0-tt_0)\al_0,(t_1-tt_1+t)\al_1,\cdots,(t_n-tt_n)\al_n,\cdots\big),$$}
 and
{\small
 $$\Psi_2 \: \big(t,(t_0\al_0,t_1\al_1,\cdots,t_n\al_n,\cdots)\big) \mapsto
    \big((t-1)\id_x,(2-t)\alpha_1, 0\id_x,\cdots,0\id_x,\cdots\big).$$}
  Here, $x$ is the common source of the $\alpha_i$.
  That $\varphi$ is  locally shrinkable follows from
  Lemma \ref{L:epibasechange0}.
\end{proof}

\begin{defn}{\label{D:classifying}}
  Let $\X$ be a topological stack. By a {\bf classifying space} for $\X$ we mean a
  topological space $X$ and a map $\varphi \: X \to \X$ such that $\varphi$ is a 
  universal weak equivalence (Definition \ref{D:stackshrinkable}). By abuse of noation, we often
  drop the $\varphi$ from the notation and call $X$ a classifying space for $\X$.
\end{defn}

Let us rephrase  the above theorem as a statement about existence of
classifying spaces for topological stacks. 

\begin{thm}[{\bf Existence of a classifying space}]{\label{T:nice}}
 Every topological stack $\X$ admits an atlas $\varphi \: X \to \X$
   which is  locally shrinkable.
   In particular, $(X,\varphi)$ is a classifying space for  $\X$.
\end{thm}

\begin{proof}
  Choose an arbitrary presentation $\bbX$ for $\X$. Then, the morphism
  $\varphi \: B\bbX \to \X$ of Proposition \ref{P:main} is the desired atlas; see
  Lemma \ref{L:shrinkable}.
\end{proof}

\begin{cor}{\label{C:Morita}}
  Every topological groupoid $[R\sst{}X]$ is Morita equivalent (\cite{Noohi}, $\S$8)
  to a topological groupoid $[R'\sst{}X']$ in which the source and target maps
  are locally shrinkable (in particular, they are universal weak equivalences).
\end{cor}

\begin{proof}
  The desired groupoid is $[B\bbX\times_{\varphi,\X,\varphi}B\bbX \sst{} B\bbX]$,
  where $\bbX=[R\sst{}X]$.
\end{proof}

\begin{rem}{\label{R:simplicial}}
   We will see in $\S$\ref{S:Htpytype} that the atlas $\varphi \: X \to
   \X$ of Theorem \ref{T:nice} can be used to define a weak homotopy
   type for the stack $\X$. We saw in Remark \ref{R:simp} that
   the diagram
   $$\xymatrix@=16pt@M=8pt{  B\bbX \ar[r]^{q} \ar[rd]_{\varphi} &
                                      |\bbX|  \ar@{..>}[d] \\
                                             & [X/R]} $$
  can not always be filled. It is, however, true in many cases that
  the map $q$ is a weak-equivalence. In such cases, the geometric realization
  $|\bbX|$ of a groupoid presentation $\bbX$ for $\X$ can very well
  be used to define the weak homotopy type of $\bbX$. The question
  that remains to be answered is, under what condition is the map
  $q\: B\bbX \to |\bbX|$
  a weak equivalence?
\end{rem}

% ------------------------------------------
\section{Some category theoretic lemmas}{\label{S:Lemmas}}

In this section, we prove a few technical results of category
theoretic nature. These results will be needed in our functorial
description of the classifying space of a topological stack
($\S$\ref{S:Htpytype}). The reader who is not interested in category
theoretical technicalities can proceed to the next section.

Throughout this section, the set up will be as follows. Let $\mfC$
be a 2-category with fiber products. Assume all 2-morphisms in
$\mfC$ are invertible. Let $[\mfC]$ be the category obtained by
identifying 2-isomorphic 1-morphisms in $\mfC$.  Let $\sfB$ be a
full subcategory of $\mfC$ which is closed under fiber products.
Assume that $\sfB$ is a 1-category, that is, there is at most one
2-morphism between every two morphisms in $\sfB$. The example 
to keep in mind is where $\mfC=\TopSt$ is the 2-category of 
topological stacks and $\sfB=\Top$ is the category of topological
spaces (realized as a subcategory of $\TopSt$ via Yoneda embedding).

Let $R$ be a class of morphisms in $\sfB$ which
 contains the identity morphisms and is closed
under base extension and 2-isomorphism. We define $\tilde{R}$ to be
the class of morphisms $f \: y \to x$ in $\mfC$ such that for every
morphism $p\: t \to x$,  $t \in \sfB$, the base extension $r$ of $f$
along $p$ belongs to $R$.
       $$\xymatrix@=16pt@M=8pt{ t\times_x y
                \ar[r] \ar[d]_r & y \ar[d]^{f} \\
                  t \ar[r]_p  & x} $$
In other words, $\tilde{R}$ is the smallest class of morphisms in 
$\mfC$ which is invariant under  base extension and
contains $R$ and all identity morphisms. The example to keep in mind is
where $R$ is the class of (weak) homotopy equivalences in $\Top$. In this
case, $\tilde{R}$ will be the class of representable universal (weak) 
equivalences in $\TopSt$. Also see Proposition \ref{P:univwe}.

\begin{lem}{\label{L:adjoint}}
  The set up being as above, assume that for
  every object $x$ in $\mfC$, there exists an object $\Theta(x)$ in
  $\sfB$ together with a morphism $\varphi_x \: \Theta(x) \to x$
  which belongs to $\tilde{R}$.   Then, the inclusion functor $\sfB
  \to [\mfC]$ induces a fully faithful functor $\iota \: R^{-1}\sfB
  \to R^{-1}[\mfC]$. Furthermore, $\Theta$ naturally  extends to a
  functor  $R^{-1}[\mfC] \to R^{-1}\sfB$ that is a right adjoint to
  $\iota$. Finally, $\Theta$ can
  be defined so that the counits of adjunction are the identity maps and
  the units of  adjunction belong to $\tilde{R}$.
\end{lem}

\begin{proof}
  In the proof we will make use of the calculus of right fractions for
  $\bar{R}$ ($\S$\ref{A:Calculus}) to describe morphisms in the
  localized categories. Here, $\bar{R}$ is the closure of $R$ under composition.
  (Notice that localization with respect to $R$ and $\bar{R}$ yields
  the same result.)

  First, let us explain  how to extend $\Theta$ to a functor. We
  will assume  $\Theta(x)=x$, whenever $x \in \sfB$.

  Given a  morphism $f \: x \to y$ in $[\mfC]$, set
  $z:=\Theta(x)\times_{y}\Theta(y)$, as in the 2-cartesian diagram
       $$\xymatrix@=16pt@M=8pt{
               z \ar[r]^(0.4){g} \ar[d]_r & \Theta(y) \ar[d]^{\varphi_{y}} \\
                  \Theta(x) \ar[r]_{f\circ \varphi_{x}} & y} $$
  By hypothesis, we have $r \in R$. We define
  $\Theta(f) \: \Theta(x) \to \Theta(y)$ to be the span $(r,g)$.
   The proof that $\Theta$ is well-defined and respects composition is straightforward.

  To prove that $\Theta$ is a right adjoint to $\iota$, we show that
  composing with the morphism $\varphi_{x} \: \Theta(x) \to x$
  induces a bijection
    $$\Hom_{R^{-1}\sfB}(t,\Theta(x)) \risom \Hom_{R^{-1}[\mfC]}(t,x)$$
  for every $t \in \sfB$. With the notation of $\S$\ref{A:Calculus},
  we have to show that the map
     $$P \: \Roof(t,\Theta(x)) \to \Roof(t,x)$$
  which sends a span $(r,g)$ to $(r,\varphi_{x}\circ g)$ induces
  a bijection
    $$\pi_0(P) \: \pi_0\Roof(t,\Theta(x)) \to
         \pi_0\Roof(t,x).$$

  We define a functor
     $$Q \: \Roof(t,x) \to \Roof(t,\Theta(x))$$
  as follows. Let $(r,g) \in \Roof(t,x)$. Then $Q(r,g)$ is
  defined to be $(r\circ\rho,g')$, as in the diagram
     $$\xymatrix@=16pt@M=8pt{
          & s\times_{x}\Theta(x) \ar[r]^{g'} \ar[d]_{\rho}
                            &  \Theta(x) \ar[d]^{\varphi_{x}} \\
                                  t  & v \ar[r]_g \ar[l]^r & x} $$
  There are natural transformations of functors
       $$\id_{\Roof(t,\Theta(x))} \Rightarrow Q\circ P$$
  and $P\circ Q \Rightarrow \id_{\Roof(t,x)}$. This is enough to
  establish that $\pi_0(P)$ and $\pi_0(Q)$ induce inverse
  bijections between $\pi_0\Roof(t,\Theta(x))$ and
  $\pi_0\Roof(t,x)$. (For example, this can be seen by noticing that
  $P$ and $Q$ induce an equivalence of categories between the
  groupoids generated by inverting all arrows in
  $\Roof(t,\Theta(x))$  and $\Roof(t,x)$.)

  Fully faithfulness of $\iota$ follows from the fact that the unit
  of adjunction $\iota\circ\Theta \Rightarrow \id_{\sfB}$ is
  an isomorphism.
\end{proof}

\begin{cor}{\label{C:eq}}
  The functors $\iota$ and $\Theta$ induce an equivalence of
  categories $R^{-1}\sfB\cong\tilde{R}^{-1}[\mfC]$. In fact,
  on the right hand side, instead of inverting $\tilde{R}$,
  it is enough to invert $R$ together with
  all the morphisms $\varphi_x \: \Theta(x) \to x$.
\end{cor}

\begin{rem}
  It can be shown that the functors $P$ and $Q$ appearing in the
  proof of Lemma \ref{L:adjoint} are indeed adjoints
    $$P \: \Roof(t,\Theta(x))   \rightleftharpoons \Roof(t,x): \!
    Q$$
   Therefore, if we work with the 2-categorical enhancements of
   $R^{-1}\sfB$ and $R^{-1}[\mfC]$ (see Remark \ref{R:enhance}), then
   the adjunction between $\iota$ and $\Theta$ can  be enhanced
   to a 2-categorical adjunction.
\end{rem}

\begin{lem}{\label{L:sub}}
      Let $F \: \sfB \rightleftharpoons \sfC: \! G$ be an
  adjunction between categories. Let $\sfB' \subset \sfB$ and
  $\sfC' \subset \sfC$ be full subcategories such that
  $F(\sfB')  \subset \sfC'$ and $G(\sfC') \subset \sfB'$. Then, the
  restriction of $F$ and $G$ to these subcategories induces an
  adjunction
   $F' \: \sfB' \rightleftharpoons \sfC': \! G'$
\end{lem}

\begin{proof}
   Obvious.
\end{proof}

\begin{lem}{\label{L:adjunction}}
  Let $F \: \sfB \rightleftharpoons \sfC: \! G$ be an
  adjunction between categories. Let $S \subset \sfB$ and
  $T \subset \sfC$ be classes of morphisms such that $F(S)\subset T$
  and $G(T)\subset S$. Then, we have an induced adjunction
    $$\tilde{F} \: S^{-1}\sfB \rightleftharpoons
                    T^{-1}\sfC: \! \tilde{G}$$
  between the localized categories. Furthermore, if $F$ is fully
  faithful, then so is $\tilde{F}$.
\end{lem}
\begin{proof}
  We will use the 2-categorical formulation of adjunction
  (\cite{MacLane}, IV.1, Theorem 2). A standard Yoneda type argument
  shows that to give an adjunction
         $$F \: \sfB \rightleftharpoons \sfC: \! G$$
  is the same thing as giving adjunctions
    $$\mathsf{K}^G \: \mathsf{K}^{\sfB} \rightleftharpoons
     \mathsf{K}^{\sfC}: \!\mathsf{K}^F,$$
   for every category $\mathsf{K}$,
   which are functorial with respect to change of $\mathsf{K}$. Here,
   $\mathsf{K}^{\sfB}$ stands for the category of functors from
   $\sfB$ to $\mathsf{K}$.

 By the universal property of localization,
 $\mathsf{K}^{S^{-1}\sfB}$ is naturally identified with a full
 subcategory of $\mathsf{K}^{\sfB}$. Similarly,
 $\mathsf{K}^{S^{-1}\sfC}$ is naturally identified with a full
 subcategory of $\mathsf{K}^{\sfC}$. The assumption that $F$
 and $G$ respect $S$ and $T$ implies that
 $\mathsf{K}^G(\mathsf{K}^{S^{-1}\sfB}) \subset
 \mathsf{K}^{S^{-1}\sfC}$ and
 $\mathsf{K}^F(\mathsf{K}^{S^{-1}\sfC}) \subset
 \mathsf{K}^{S^{-1}\sfB}$.
 So, we have an induced adjunction
     $$\mathsf{K}^G \: \mathsf{K}^{S^{-1}\sfB} \rightleftharpoons
     \mathsf{K}^{S^{-1}\sfC} : \!\mathsf{K}^F.$$
 Since these adjunctions are functorial with respect to $K$, the
 Yoneda argument gives the desired adjunction
 $\tilde{F} \: S^{-1}\sfB \rightleftharpoons  T^{-1}\sfC: \! \tilde{G}$.

 The statement about fully faithfulness follows from the fact that
 a left adjoint $F$ is fully faithful if the unit of adjunction is
 an isomorphism of functors.
\end{proof}

% ------------------------------------
\subsection{Lemma \ref{L:adjoint} for diagram categories}{\label{SS:diagram}}

We prove a version of Lemma \ref{L:adjoint} for diagrams. The set up
will be as in Lemma \ref{L:adjoint}. We will assume, in addition,
that $R$ is closed under fiber products. This means that, given two
morphisms $Y \to X$ and $Z \to X$ in $R$, the fiber product
$Z\times_{X} Y \to X$ is also in $R$. Since $R$ is closed under base
extension, this is automatic if $R$ is closed under composition.

Suppose we are given a category $\sfD$, which we think of as a
diagram. We assume either of  the following holds:\footnote{The
common feature of these two conditions, which is all we need to
prove our lemma, is that $R$ is closed under products indexed by any
set whose cardinality is less that or equal to the degree (i.e., the
number of arrows coming out) of some object in $\sfD$.}

\begin{itemize}
 \item[$\mathbf{A.}$] The category $\sfD$ has the property that
  every object $d$ in $\sfD$ has finite ``degree'', that is,
  there are only finitely many arrows coming out of $d$; or,
 \item[$\mathbf{B.}$] The class $R$ is closed under arbitrary fiber
  products.
\end{itemize}

We denote the 2-category of lax functors from $\sfD$ to $\mfC$ by
$\mfC^{\sfD}$. We think of objects in $\mfC^{\sfD}$ as diagrams in
$\mfC$ indexed by $\sfD$.

\begin{lem}{\label{L:diagramadjoint}}  % WAS P:diagramunivwe
  Let $T$ (resp., $\tilde{T}$) be the class of all transformations $\tau$ in the
  diagram category $[\mfC^{\sfD}]$ which have the property that for
  every $d \in \sfD$ the corresponding morphism $\tau_d$ in $\mfC$
  is in $R$ (resp., $\tilde{R}$). Then, the  inclusion functor $\sfB^{\sfD} \to
  [\mfC^{\sfD}]$ induces a fully faithful functor $\iota \:
  T^{-1}\sfB^{\sfD} \to T^{-1}[\mfC^{\sfD}]$,  and
  $\iota$ has  a right adjoint $\Theta^{\sfD}$.
  Furthermore, $\Theta^{\sfD}$ can be defined so that the counits
  of adjunction are the identity transformations and
  the units of adjunction are honest transformations in $\tilde{T}$.
\end{lem}

\begin{proof}
 We use Lemma \ref{L:adjoint} with $\mfC$ and $\sfB$ replaced by the
 corresponding diagram categories. We have to verify that for every
 functor $p \: \sfD \to \mfC$, there exists a functor
 $\Theta^{\sfD}(p) \: \sfD \to \sfB$ together with a  natural
 transformation of functors $\varphi_p \: \Theta^{\sfD}(p)
 \Rightarrow p$ such that every morphism $\varphi_{p,d}$ in this
 transformation is in $\tilde{R}$. To show this, we will make use of
 the relative Kan extension of
 $\S$\ref{A:Kan}. Let us fix the set up first.

 Let $\mfU$ be the 2-category fibered over $\mfC$ whose fiber
 $\mfU(x)$ over  an object $x\in \mfC$  is the 2-category of
 morphisms $h \: a \to x$  in $\tilde{R}$. More precisely,
 the objects in $\mfU$ are morphisms $a \to x$ in $\tilde{R}$.
 The morphism in $\mfU$
 are 2-commutative squares
              $$\xymatrix@=6pt@M=6pt{
         b   \ar[rr]   \ar[dd] & & a \ar[dd]   \\
                 \ar@{=>}[ur]_(0.4){\tau}     &  & \\
   y   \ar[rr]  & &  x  }$$
 in $\mfC$ whose vertical arrows are in $\tilde{R}$; such a morphism
 is defined to be cartesian if the square
 is 2-cartesian.  The 2-morphisms in $\mfU$ are defined in the obvious way.

 It follows that morphisms in $\mfU(x)$ are
 2-commutative triangles in $\mfC$. It is easy to see that
 $\mfU(x)$ is indeed a 1-category and not just a 2-category. The
 functor $\pi \: \mfU \to \mfC$ is the forgetful functor which sends
 $h \: a \to x$ to $x$. The pull-back functor $f^{\Box} \: \mfU(y) \to
 \mfU(x)$ for a morphism $f \: x \to y$ is the base extension along
 $f$.

 Let $\sfE$ be the discrete category
 with the same set of objects as $\sfD$, and let $\sth \: \sfE
 \to \sfD$ be the functor which sends an object to itself.
 Either of the conditions
 ($\mathbf{A}$) or  ($\mathbf{B}$) above implies that $\pi \:
 \mfU \to \mfC$ is $\sth$-complete (Definition \ref{D:relativecomplete})
 at every $p \: \sfD \to \mfC$.

 We can now define $\Theta^{\sfD}(p)$ and $\varphi_p$ as follows. Let $p \: \sfD \to
 \mfC$ be a diagram in $\mfC$.  For every object $d \in \sfD$, we
 denote $p(d)$ by $x_d$. For each $d$, choose a map $\varphi_d \:
 \Theta(x_d) \to x_d$,  with $\Theta(x_d)$ in $\sfB$ and $\varphi_d$
 in $\tilde{R}$.
 This gives a functor $P \: \sfE \to \mfU$, $d \mapsto \varphi_d$,
 which lifts $p$. By Proposition \ref{P:Kan}, we have a right Kan
 extension $R\sth(P) \: \sfD \to \mfU$. To give such a functor
 $R\sth(P)$ is the same thing as giving a functor
 $\Theta^{\sfD}(p) \: \sfD \to \mfC$ together with a natural
 transformation of functors $\varphi_p \: \Theta^{\sfD}(p)
 \Rightarrow p$. More precisely, for every $d \in
 \sfD$, we define $\Theta^{\sfD}(p)(d)$ and $\varphi_{p,d}$ by
   $$R\sth(P)_d=\ \ \Theta^{\sfD}(p)(d) \llra{\varphi_{p,d}}
                                x_d \ \ \in \mfU(x_d).$$

 All that is left to check is that $\Theta^{\sfD}(p)$  factors
 through $\sfB$. To see this, note that, by the construction of the
 Kan extension (see proof of Proposition \ref{P:Kan}),
 $\Theta^{\sfD}(p)(d)$ is the product in $\mfU(d)$ of a family of
 objects $\{y_0,y_1,\cdots\}$ in $\mfU(d)$, one of which, say $y_0$,
 is $\Theta(x_d)$. (Note the abuse of notation: each $y_i$ is
 actually a morphism $y_i \to x_d$.) Denote the product of the rest
 of the objects by $y$. So $\Theta^{\sfD}(p)(d)=\Theta(x_d)\times
 y$, the product being taken in $\mfU(d)$. Note that product in the
 fiber category $\mfU(d)$ is calculated by taking fiber product
 over $x_d$ in $\mfC$. Hence, the following diagram is
 2-cartesian in $\mfC$.
       $$\xymatrix@=16pt@M=8pt{
           \Theta^{\sfD}(p)(d) \ar[r] \ar[d] & y \ar[d]^h \\
                         \Theta(x_d)  \ar[r]_{\varphi_d} & x_d} $$
 Since $h \: y \to x_d$ is in $\mfU(d)$, its base extension  along
 $\varphi_d$ is in $R$. That is, $\Theta^{\sfD}(p)(d)$ is in $\sfB$
 and $\Theta^{\sfD}(p)(d) \to \Theta(x_d)$ is a
 morphism in $R$. This shows that $\Theta^{\sfD}(p)(d)$ is in $\sfB$,
 which is what we wanted to prove.
\end{proof}

While proving Lemma \ref{L:diagramadjoint} we have also proved the
following.

\begin{cor}{\label{C:honest}}
  Let $\mfS$ be a small sub 2-category of $\mfC$, and denote the
  inclusion of $\mfS$ in $\mfC$ by $i_{\sfS}$. Then, there is a
  functor $\Theta_{\mfS} \: \mfS \to \sfB$ together with a natural
  transformation $\varphi_{\mfS} \: \Theta_{\mfS} \Rightarrow
  i_{\mfS}$ such that for every $x \in
  \mfC$, $\varphi_{\mfS}(x) \: \Theta_{\mfS}(x) \to x$ is in
  $\tilde{R}$.
\end{cor}

\begin{proof}
  Let $\sfD=[\mfS]$ and think of it as a diagram category. Think of
 $p=[i_{\mfS}] \: \sfD \to [\mfC]$ as a $\sfD$-diagram in $[\mfC]$.
 Then, with the notation of  Lemma \ref{L:diagramadjoint}, the
 sought after $\Theta_{\mfS}$ and $\varphi_{\mfS}$ are exactly
 $\Theta^{\sfD}([i_{\mfS}])$ and $\varphi_{[i_{\mfS}]}$ (precomposed
 with the projection $\mfS\to[\mfS]$).
\end{proof}

% -------------------------------------------------------------------
\section{Functorial description of the classifying space}{\label{S:Htpytype}}

Theorem \ref{T:nice} is  saying that every topological stack has a
classifying space. In this section, we use the category theoretic
lemmas of $\S$\ref{S:Lemmas} to give a functorial formulation of
this fact (Theorem \ref{T:whtpytype} and Theorem \ref{T:htpytype}).

\begin{prop}{\label{P:univwe}}
  Let $R \subset \Top$ be the class of locally shrinkable maps
  (Definition \ref{D:retraction}). Then, the inclusion functor $\Top
  \to \TopSt$ induces a fully faithful functor $\iota \: R^{-1}\Top
  \to R^{-1}[\TopSt]$, and $\iota$ has a right adjoint $\Theta \:
  R^{-1}[\TopSt] \to R^{-1}\Top$. Furthermore, $\Theta$ can
  be defined so that the counits of adjunction are the identity maps and
  the units of adjunction are honest morphisms of topological stacks
  which are locally shrinkable.
\end{prop}

\begin{proof}
 Apply Lemma \ref{L:adjoint} to the inclusion $\Top \to \TopSt$ with
 $R$ being the class of locally shrinkable maps of topological {\em
 spaces}. For every topological stack $\X$, the existence of a
 topological space $\Theta(\X)$ which satisfies the requirement of
 Lemma \ref{L:adjoint} is guaranteed by Theorem \ref{T:nice}.
\end{proof}

% -------------------------------------------------
\subsection{Functorial description of the classifying space}

The following theorem implies that  the 
classifying space of a topological stack $\X$ is functorial in $\X$. 
In particular, the classifying space of a topological stack can be used
to define a weak homotopy type for it.

\begin{thm}{\label{T:whtpytype}}
 Let $S_{w.e.}$ be the class of weak equivalences in $\Top$. Let
 $\Top_{w.e.}:=S_{w.e.}^{-1}\Top$ be the category of weak homotopy
 types. Then, 
 the inclusion functor $\Top \to \TopSt$ induces a
 fully faithful functor $\iota \: \Top_{w.e.} \to
 S_{w.e.}^{-1}[\TopSt]$, and $\iota$ has a right adjoint $\Theta \:
 S_{w.e.}^{-1}[\TopSt] \to\Top_{w.e.} $. Furthermore, $\Theta$ can be
 defined so that the counits of adjunction are the identity maps and
 the units of adjunction are honest morphisms of
 topological stacks  which are  locally shrinkable.
\end{thm}

\begin{proof}
   Consider the adjunction $\iota \: R^{-1}\Top \rightleftharpoons
  R^{-1}[\TopSt] :\!\Theta$ of Proposition \ref{P:univwe}. Let $S
  \subset R^{-1}\Top$ be the class of weak equivalences and set
  $T=\iota(R)$. It is easy to see that the conditions of Lemma
  \ref{L:adjunction} are satisfied. This gives us the desired
  adjunction $\iota \: \Top_{w.e.} \rightleftharpoons
  S_{w.e.}^{-1}[\TopSt] :\!\Theta$. (By abuse of notation, we have
  denoted the induced functors on localized categories again by
  $\iota$ and $\Theta$.)
\end{proof}

The functor $\Theta \: S_{w.e.}^{-1}[\TopSt] \to \Top_{w.e.}$ should
be thought of as a functor that associates to every topological
stack its {\bf weak homotopy type}.

We say that a morphism $f \: \X \to \Y$ of topological stacks is a
{\bf weak homotopy equivalence}, if $\Theta(f)$ is so. Let
$\TopSt_{w.e.}$ be the localization of the category $[\TopSt]$ of
topological stacks with respect to weak equivalences.

\begin{cor}
The functors $\iota$ and $\Theta$ of Theorem \ref{T:whtpytype}
induce an equivalence of categories
      $$\Top_{w.e.}\cong\TopSt_{w.e.}.$$
In fact, the category on the right can be obtained by inverting
$S_{w.e.}$ and all locally shrinkable morphisms of topological
stacks.
\end{cor}

\begin{proof}
  Immediate from Theorem \ref{T:whtpytype} (also see Corollary
  \ref{C:eq}).
\end{proof}

% -------------------------------------------------
\subsection{The homotopy type of a \hpara\ topological
stack}{\label{SS:homotopically}

For the class of {\em \hpara} topological stacks (Definition
\ref{D:paracompact}) we can strengthen Theorem \ref{T:whtpytype} by
showing that every such stack has a natural homotopy type (Theorem
\ref{T:htpytype}).

\begin{defn}{\label{D:paracompact}}
  We say that a topological stack $\X$ is {\bf \hpara}
  if there exists a  parashrinkable map $\varphi \: X \to \X$  
  (Definition \ref{D:retraction}) with
  $X$ a paracompact topological space.
\end{defn}

For instance, if the atlas $\varphi \: X \to \X$  of Theorem
\ref{T:nice} can be chosen so that $X$ is paracompact, then $\X$ is
\hpara.

\begin{prop}{\label{P:paracompact}}
Let $\X$ be the quotient stack of a topological groupoid
$\bbX=[X_1\sst{} X_0]$. In each of the following cases, $B\bbX$ is
paracompact, hence $\X$ is \hpara:

\begin{itemize}
  \item[$\mathbf{1.}$] The spaces $X_1$ and $X_0$ are regular and
  Lindel\"of. (A space $X$ is {\em Lindel\"of} if every open cover
  of $X$ has a countable subcover. A space $X$ is {\em regular}
  if every closed set can be separated from every point by open
  sets.)

  \item[$\mathbf{2.}$] The spaces $X_1$ and $X_0$ are metrizable.

  \item[$\mathbf{3.}$] The space $X_1$ and $X_0$ are paracompact,
  Hausdorff, and they admit a proper surjective map from a metric
  space. % Nagata, modern topology, pages 299-302.
\end{itemize}

\end{prop}

The above proposition was also (independently) observed by Johaness
Ebert. We will omit the proof here. The key point is that, under the
given assumptions, the multiple fiber products
$X_n:=X_1\times_{X_0}\times\cdots\times_{\X_0}\times X_1$ are again
paracompact. This is not necessarily true if $X_1$ and $X_0$ are
only assumed to be paracompact, because the products of two
paracompact spaces is not necessarily paracompact. For this reason,
we have to replace the paracompactness requirement on $X_1$ and
$X_0$ by something stronger which is closed under products.

Let $\Para$ be the category of paracompact topological spaces and
$\HPTopSt$ the category of \hpara\ topological stacks. Let $S_{h.e.}
\subset \Para$ be the class of homotopy equivalence. Let
$\Para_{h.e.}=S_{h.e.}^{-1}\Para$ be the category of paracompact
homotopy types.

\begin{rem}{\label{R:he}}
 There is an alternative way of describing the categories $\Para_{h.e.}$ and
 $S_{h.e.}^{-1}[\HPTopSt]$ which is perhaps more natural.
 Let us give this description in the case of
 $S_{h.e.}^{-1}[\HPTopSt]$.

 The objects of $S_{h.e.}^{-1}[\HPTopSt]$ are the same as
 those of $[\HPTopSt]$. The morphisms $\Hom_{S_{h.e.}^{-1}[\HPTopSt]}(\X,\Y)$ are
 obtained from $\Hom_{\HPTopSt}(\X,\Y)$ by passing to a certain
 equivalence relation. If $\X$ is a topological space, this
 relation is just the usual homotopy between maps (defined via cylinders). If $\X$ is
 not an honest topological space, then two morphism $f,g \: \X \to
 \Y$ in $[\HPTopSt]$ get identified in $\Hom_{S_{h.e.}^{-1}[\HPTopSt]}(\X,\Y)$ if
 and only if for every map $h \: T \to \X$ from a topological space $T$,
 the compositions $f\circ h$ and $g\circ h$ are equivalent.

 It is easy to verify that this category satisfies the universal
 property of localization.
\end{rem}

\begin{lem}{\label{L:paracompact}}
  Let $R$ be as in Proposition \ref{P:univwe}. Then, the inclusion
  functor $\Para \to \Top$ induces
  a fully faithful functor $\Para_{h.e.} \to R^{-1}\Top$.
  In the stack version, the functor $S_{h.e.}^{-1}[\HPTopSt] \to R^{-1}[\TopSt]$
  may no longer be fully faithful, but it induces a bijection
    $$\Hom_{S_{h.e.}^{-1}[\HPTopSt]}(T,\Y) \cong \Hom_{R^{-1}[\TopSt]}(T,\Y)$$
  whenever $T$ is a topological space.
\end{lem}

\begin{proof}
   We only prove the statement for $[\HPTopSt]$. The case of $\Para$
   is proved similarly.

   Before proving the bijectivity, let us explain why
   we have an induced functor
     $$S_{h.e.}^{-1}[\HPTopSt] \to R^{-1}[\TopSt]$$
   in the first place. By the discussion of Remark \ref{R:he}, morphisms in
   $S_{h.e.}^{-1}[\HPTopSt]$ are obtained from those of $[\HPTopSt]$ by passing to
   a certain equivalence relation. It is easy to check that the
   localization functor $[\HPTopSt] \to R^{-1}[\HPTopSt]$ sends an
   entire   equivalence class of morphisms to one morphism. (This
   follows from the fact that the projection map
   $X\times[0,1] \to X$ of a cylinder is in $R$.) Therefore, we have
   a functor $S_{h.e.}^{-1}[\HPTopSt] \to R^{-1}[\TopSt]$.

   Now, let $T$ be a paracompact topological space and $\Y$ a
   \hpara\ topological stack. We want to show that
    $$\gamma \: \Hom_{S_{h.e.}^{-1}[\HPTopSt]}(T,\Y)
                         \to \Hom_{R^{-1}[\TopSt]}(T,\Y)$$
   is a bijection.

   Let $\bar{R}$ be the class of morphisms in $\TopSt$ which are compositions of
   finitely many locally shrinkable morphisms. By
   $\S$\ref{A:Calculus}, morphisms in
   $R^{-1}[\TopSt]=\bar{R}^{-1}[\TopSt]$ can be calculated using a
   calculus of right fractions. By Lemma \ref{L:zigzag},
     $$\Hom_{R^{-1}[\TopSt]}(T,\Y)=\Hom_{[\TopSt]}(T,\Y)/_{\sim},$$
   where $\sim$ is the equivalence relation generated by $\bar{R}$
   -homotopy ($\S$\ref{A:Calculus}). On the other hand
     $$\Hom_{S_{h.e.}^{-1}[\HPTopSt]}(T,\Y)=\Hom_{[\TopSt]}(T,\Y)/_{\sim'},$$
   where $\sim'$ is the usual homotopy (Remark \ref{R:he}). To
   complete the proof, we show that $\sim$ and $\sim'$ are the same.
   That is, two honest morphisms $f,f' \: T \to \Y$ are homotopic if
   and only if they are $\bar{R}$-homotopic. First assume that $f$
   and $f'$ are $\bar{R}$-homotopic and consider an
   $\bar{R}$-homotopy
   ($\S$\ref{A:Calculus}) between them, as in the diagram
          $$\xymatrix@=45pt@M=4pt{
           T \ar@/^/ [r]^{t} \ar@/^2pc/[rr]_{f}
             \ar@/_/ [r]_{t'} \ar@/_2pc/[rr]^{f'}
            &  V \ar[l] |-{r} \ar[r]^g & \Y.   }$$
  Since $r$ is a composition of locally shrinkable morphisms and $T$
  is paracompact, it follows from Lemma \ref{L:lift}, together with
  the fact that $T\times [0,1]$ is paracompact, that $t$ and $t'$
  are homotopic. This implies that $f$ and $f'$ are also homotopic.
  Conversely, assume $f$ and $f'$ are homotopic. Then, we can form
  an $\bar{R}$-homotopy diagram between them by taking $V=T\times
  [0,1]$.  The proof is complete.
\end{proof}

The following theorem says that every \hpara\ topological stack has
a natural homotopy type.

\begin{thm}{\label{T:htpytype}}
 The inclusion functor $\Para \to \HPTopSt$ induces a fully
 faithful functor $\iota \: \Para_{h.e.} \to S_{h.e.}^{-1}[\HPTopSt]$,
 and $\iota$ has a right adjoint $\Theta$. Furthermore, the right adjoint
 $\Theta$ can be defined so that the counits of adjunction are the
 identity maps and the units of adjunction are honest morphisms of
 topological stacks which are locally shrinkable.
\end{thm}

\begin{proof}
    Consider the adjunction $\iota \: R^{-1}\Top \rightleftharpoons
    R^{-1}[\TopSt] :\!\Theta$ of Proposition \ref{P:univwe}. We can
    arrange so that for every \hpara\ $\X$,
    $\Theta(\X)$ is paracompact.

    By Lemmas \ref{L:sub} and \ref{L:paracompact} , if  in both sides
    of the adjunction we restrict to the paracompact objects, we obtain the
    adjunction
         $$\iota \: \Para_{h.e.} \rightleftharpoons
    S_{h.e.}^{-1}[\HPTopSt] :\!\Theta,$$
    which is what we were after.
\end{proof}

The functor $\Theta \: S_{h.e.}^{-1}[\HPTopSt] \to \Para_{h.e.}$
should be thought of as a functor that associates to every \hpara\
topological stack its {\bf homotopy type}.

We say that a morphism $f \: \X \to \Y$ of \hpara\ topological
stacks is a {\bf  homotopy equivalence}, if $\Theta(f)$ is so. Let
$\HPTopSt_{h.e.}$ be the localization of the category of
hoparacompact stacks with respect to homotopy equivalences. We have
the following.

 \begin{cor}
  The functors $\iota$ and $\Theta$ of Theorem \ref{T:htpytype} induce
  an equivalence of categories
    $$\Para_{h.e.}\cong\HPTopSt_{h.e.}.$$
  In fact, the category on the right can be obtained by inverting $S_{h.e.}$
  and all locally shrinkable morphisms of hoparacompact topological stacks.
\end{cor}

\begin{proof}
  Immediate from Theorem \ref{T:htpytype} (also see Corollary
  \ref{C:eq}).
\end{proof}

% ---------------------------------------------------------------
\section{Homotopical stacks}{\label{S:Homotopical}}

In the previous sections, we proved the existence of a functorial
classifying for an arbitrary topological stack. Although this may
seem to be quite general, there are still some important classes of
stacks to which this may not apply. The first example that
comes to mind is the mapping stack $\Map(\Y,\X)$ of two topological
stacks. It can be shown (see \cite{Mapping}) that if $\Y=[Y_0/Y_1]$, for a
topological groupoid $[Y_1\sst{}Y_0]$ with $Y_0$ and $Y_1$ compact,
then the mapping stack $\Map(\Y,\X)$ is again topological. The
compactness condition on $\Y$, however, is quite restrictive and it
seems that without it $\Map(\Y,\X)$ may not be topological in
general (but we do not have a counter example).

Nevertheless, we prove in \cite{Mapping} that when $Y_0$ and $Y_1$
are locally compact and $\X$ is arbitrary, $\Map(\Y,\X)$ is not far
from being topological. More precisely, it is {\em paratopological}
in the sense of Definition \ref{D:paratop} below. Every topological
stack is, by definition, paratopological.

In this section, we show that our construction of the homotopy types
for topological stacks can be extended to a larger class of stacks
called {\em homotopical stacks} (Definition \ref{D:paratop}).
Homotopical stacks include all paratopological stacks (hence, in
particular, all topological stacks).

\begin{defn}{\label{D:paratop}}
 We say that a stack $\X$ is {\bf paratopological} if it 
 satisfies the following conditions:
 \begin{itemize}
   \item[$\mathbf{A1.}$] Every map $T \to \X$ from a topological
     space $T$ is representable (equivalently, the diagonal
     $\X \to \X\times\X$ is representable);

   \item[$\mathbf{A2.}$] There exists a
      morphism $X \to \X$ from a topological space $X$ such that for
      every morphism $T \to \X$, with $T$ a paracompact topological
      space, the base extension $T\times_{\X} X
       \to T$ is an epimorphism of topological spaces (i.e., admits local sections)
 \end{itemize}
   If ($\mathbf{A2}$) is only satisfied with  $T$ a CW complex, then
   we say that $\X$ is {\bf pseudotopological}. If in
   ($\mathbf{A2}$) we require that the base extensions to be
   weak equivalences, we say that $\X$ is {\bf homotopical}.\footnote{
   Remark that if the latter condition is satisfied for all CW
   complexes $T$ then it is satisfied for all topological spaces $T$.} In this case, 
   we say   that $X$ is a {\bf classifying space} for $\X$. 
\end{defn}

  Roughly speaking, a stack $\X$ being paratopological means that,
  in the eye of a paracompact topological space $T$, $\X$ is as good
  as a topological stack (Lemmas \ref{L:A2} and \ref{L:paratop}).
  Thus, it is not surprising that the
  homotopy theory of topological stacks can be extended to
  paratopological (or even pseudotopological) stacks.

 Paratopological stacks form a full sub 2-category of the 2-category of 
 stacks which we denote by $\ParSt$. The 2-categories $\HoSt$ and $\PsSt$
 of homotopical stacks and pseudotopological stack are defined similarly.

\begin{lem}{\label{L:A2}}
 Let $\X$ be a stack over $\Top$ such that the diagonal $\X \to
 \X\times\X$ is representable. Then, $\X$ is paratopological
 (respectively, pseudotopological) if and only if there exists a
 topological stack $\bar{\X}$ and a morphism $p \: \bar{\X} \to \X$
 such that for every paracompact topological space $T$
 (respectively, every CW complex $T$), $p$ induces an equivalence of
 groupoids  $\bar{\X}(T) \to \X(T)$.
\end{lem}

\begin{proof}
 We only prove the statement for paratopological stacks. The case of
 pseudotopological stacks is similar.

 Suppose that such a map $\bar{\X} \to \X$ exists. Take an atlas $X \to
 \bar{\X}$ for $\bar{\X}$. It is clear that the composite map $X \to
 \X$ satisfies ($\mathbf{A2}$) of Definition \ref{D:paratop}.

 Conversely, assume $\X$ is paratopological and pick a map  $X \to
 \X$ as in Definition \ref{D:paratop}, ($\mathbf{A2}$). Set $X_0:=X$
 and $X_1:=X\times_{\X }X$. It is easy to see that the quotient
 stack $\bar{\X}:=[X_0/X_1]$ of the topological groupoid
 $[X_1\sst{}X_0]$ has the desired property.
\end{proof}

\begin{lem}{\label{L:paratop}}
  Let $\X$ be a paratopological (resp., pseudotopological)
  stack. Let $p \: \bar{\X} \to \X$ be as in Lemma \ref{L:A2}.
  Then, for any map $T \to \X$, with $T$ a
  paracompact topological space (resp., a CW complex), the
  base extension  $p_T \: T\times_{\X}\bar{\X} \to T$ is a homeomorphism.
\end{lem}

\begin{proof}
  We only prove the statement for paratopological stacks. The case of
  pseudotopological stacks is  similar.

  Set $\Y:=T\times_{\X}\bar{\X}$, and let $\Ym$ be its coarse
  moduli space (\cite{Noohi}, $\S$4.3). By (\cite{Noohi},
  Proposition 4.15.iii), we have a continuous map $f \: \Ym \to T$.
  Since a point is paracompact, it follows from the definition that $f$
  is a bijection. On the other hand, since $T$ is
  paracompact, $f$ admits a section. Therefore, $f$ is a
  homeomorphism.

  Since a point is paracompact, it follows from the property of the
  map $p$ that the inertia groups of $\Y$ are trivial. That is,
  $\Y$ is a quasitopological space in the sense of (\cite{Noohi},
  $\S$7, page 27). Since the coarse moduli map $\Y \to \Ym=T$ admits
  a section, it follows from (\cite{Noohi}, Proposition 7.9) that
  $\Y=\Ym=T$.
\end{proof}

A slightly weaker version of Theorem \ref{T:nice} is true for
paratopological and pseudotopological stacks.

\begin{prop}{\label{P:nice2}}
  Let $\X$ be a paratopological (resp., pseudotopological)
  stack. Then, there exists a parashrinkable
  (resp., a pseudoshrinkable) morphism $\varphi \: X \to \X$
  from a topological space $X$. In particular, $\X$ is a homotopical
  stack and $(X,\varphi)$ is a classifying space
  for it (Definition \ref{D:paratop}). 
\end{prop}

\begin{proof}
  Let $\X$ be a paratopological (resp., pseudotopological) stack,
  and let $p \: \bar{\X} \to \X$ be an approximation for it by a
  topological stack $\bar{\X}$ as in Lemma \ref{L:A2}. Choose an
  atlas $\bar{\varphi} \: X \to \bar{\X}$ for it which is locally
  shrinkable (Theorem \ref{T:nice}). Then, the composite
  $\varphi:=p\circ\bar{\varphi} \: X \to \X$ is parashrinkable
  (resp., pseudoshrinkable) by Lemma \ref{L:paratop}.
\end{proof}

\begin{rem}
Notice that, in contrast with Theorem \ref{T:nice},  the map
$\varphi$ in Proposition \ref{P:nice2} need not be an epimorphism.
\end{rem}

\begin{cor}{\label{C:pseudo}}
  We have the following full inclusions of 2-categories:
      $$\TopSt \subset \ParSt \subset \PsSt \subset \HoSt.$$
\end{cor}

\begin{proof}
  The first two inclusions are clear from the definition. The last
  inclusion follows from Proposition \ref{P:nice2}.
\end{proof}

\begin{thm}{\label{T:weakly}}
  Theorem \ref{T:whtpytype} remains valid if we replace $\TopSt$ with
  $\ParSt$ (resp., $\PsSt$ or $\HoSt$). The last
  statement in Theorem \ref{T:whtpytype} on the units and counits of
  the adjunction also remains valid provided that we replace
  locally shrinkable by parashrinkable (resp., pseudoshrinkable,
  universal weak equivalence).
\end{thm}

\begin{proof}
  The same argument used in the proof of Theorem \ref{T:whtpytype}
  carries over verbatim (we have to use  Proposition
 \ref{P:nice2} instead of Theorem \ref{T:nice}).
\end{proof}

\begin{defn}{\label{D:hpara}}
  We say that a paratopological stack  is {\bf \hpara}, if there
  exists a parashrinkable morphism $\varphi \: X \to \X$ such that
  $X$ is a paracompact topological space (see Proposition
  \ref{P:nice2}). We denote the full subcategory of $\ParSt$
  consisting
  of \hpara\   paratopological stacks  by $\HPParSt$.
\end{defn}

See Proposition \ref{P:paracompact} for examples of \hpara\ stacks.

The following theorem says that a \hpara\ paratopological stack has
the homotopy type of a paracompact topological space.

\begin{thm}{\label{T:hweakly}}
  Theorem \ref{T:htpytype} remains valid if we replace $\HPTopSt$ by
  $\HPParSt$. Furthermore, the last
  statement in Theorem \ref{T:whtpytype} on the units and counits of
  the adjunction remains valid with locally shrinkable replaced
  by parashrinkable.
\end{thm}

\begin{proof}
  The proof of Theorem \ref{T:htpytype} carries over.
\end{proof}

\begin{rem}{\label{R:pseudo}}
 There is also a version of Theorem  \ref{T:hweakly} for
 pseudotopological stacks. We define a pseudotopological
 stack to be {\em  \hcw} if there exists a
 pseudoshrinkable morphism $\varphi \: X \to \X$
 such that $X$ is a CW complex (see Proposition
 \ref{P:nice2}). We leave it to the reader to reformulate Theorem
 \ref{T:hweakly} accordingly. It follows that a \hcw\
 pseudotopological stack has the homotopy type of a CW complex.
\end{rem}

The following lemmas will be used later on when we discuss homotopy
types of diagrams of stacks.

\begin{lem}{\label{L:lim1}}
   The 2-category of stacks $\X$ whose diagonal $\X \to \X\times\X$
  is representable is closed under arbitrary (2-categorical) limits.
\end{lem}

\begin{proof}
  We prove a more general fact. Consider two diagrams
  $\bbX=\{\X_d\}$ and $\bbY=\{\Y_d\}$ of stacks, where $d$ ranges in
  some index category $\sfD$. Let $\Delta\: \bbX \Rightarrow \bbY$
  be a natural transformation such that for every $d$, $\Delta_d \:
  \X_d \to \Y_d$ is representable. We claim that the induced morphism
  $\liminv\Delta \: \liminv\bbX \to \liminv\bbY$ is also
  representable. (Applying this to the case where
  $\Y_d=\X_d\times\X_d$ and $\Delta_d \: \X_d \to \X_d\times\X_d$
  are the diagonal morphisms proves the lemma.)

  To prove the claim, take an arbitrary map $f\: T \to \liminv\bbY$
  from a topological space $T$. Note that to give $f$ is the same
  thing as to give a compatible family of maps $f_d \: T \to \Y_d$.
  It follows easily from the universal property of limits that we
  have a natural isomorphism
    $$\liminv f_d^*(\X_d) \cong f^*(\liminv\bbX).$$
  (By $f^*$ we mean pull-back along $f$, e.g.,
  $f_d^*(\X_d):=T\times_{f_d,\Y_d,\Delta_d}\X_d$.) Since the diagram
  on the left is a diagram of topological spaces, it follows that
  $f^*(\liminv\bbX)$ is a topological space.
\end{proof}

\begin{lem}{\label{L:lim2}}
  Let $\Y$ be a stack whose diagonal $\Y \to \Y\times\Y$ is
  representable. Let $p_i \: \X_i \to \Y$, $i \in I$, be a family of
  stacks over $\Y$. If every $\X_i$ is paratopological
  (resp., pseudotopoloical), then so is their fiber product
  $\prod_{\Y} \X_i$. If $I$ is finite, the same statement is true for
  topological stacks and homotopical stacks. 
\end{lem}

\begin{proof}
  We prove the case of paratopological stacks. The other cases are
  proved similarly.

  Condition ($\mathbf{A1}$) of Definition \ref{D:paratop} is
  satisfied by Lemma \ref{L:lim1}. Choose $\varphi_i \: X_i \to
  \X_i$ as in Proposition \ref{P:nice2}. Set $X=\prod_{\Y} X_i$. We
  claim that the induced map $\varphi \: X \to \prod_{\Y} \X_i$ is
  parashrinkable (hence, satisfies condition ($\mathbf{A2}$)
  of Definition \ref{D:paratop},
  $\mathbf{A2}$). Let $g \: T \to \prod_{\Y} \X_i$ be a map from a
  paracompact topological space $T$, and denote its $i$-th component
  by $g_i \: T \to \X_i$. Let $f_i \: T_i \to T$ be the base
  extension of $\varphi_i$ along $g_i$, as in the 2-cartesian diagram
      $$\xymatrix@=16pt@M=8pt{
           T_i \ar[r] \ar[d]_{f_i} & X_i \ar[d]^{\varphi_i} \\
                              T \ar[r]_{g_i} & \X_i} $$
   We have a 2-cartesian diagram
       $$\xymatrix@=16pt@M=8pt{
          \prod_{T}T_i \ar[r] \ar[d]_f & \prod_{\Y}X_i \ar[d]^{\varphi} \\
                              T \ar[r]_{g} & \prod_{\Y}\X_i }$$
   Since $f_i$ is shrinkable for every $i$, the claim follows from Lemma \ref{L:prod}.
\end{proof}

\begin{lem}{\label{L:lim3}}
  The 2-categories $\TopSt$, $\ParSt$, $\PsSt$, and $\HoSt$ are
  closed under finite limits.
\end{lem}

\begin{proof}
  To show that these 2-categories are closed under finite limits, it
  is enough that 2-fiber products exist, which is the case
  by Lemma \ref{L:lim2}. The case of an arbitrary finite limit then
  follows from the general fact that a 2-category that has 2-fiber
  products and a final object is closed under arbitrary
  finite limits.
\end{proof}

% -------------------------------------------------------
\section{Homotopy groups of homotopical stacks}{\label{S:Homotopy}}

By (\cite{Noohi}, $\S$17) we know that if $(\X,x)$ is a pointed {\em
Serre}\footnote{In [ibid.] we call these {\em topological stacks}.}
topological stack (Definition \ref{D:Serre}),  the standard
definition
  $$\pi_n(\X,x):=[(S^n,\bullet),(\X,x)]$$
of homotopy groups in terms of homotopy classes of pointed maps
gives rise to well-defined homotopy groups for $\X$ that enjoy the
expected properties. 
Theorem \ref{T:nice} gives an alternative  definition of higher
homotopy groups that works for an arbitrary topological stack $\X$.
(In fact, all we need  for this definition to make sense is that $\X$ 
be a homotopical stack.) In this section
we prove that these two definition are equivalent (Theorem \ref{T:homotopy}).

Let us first recall the definition of a Serre topological stack.

 \begin{defn}[\oldcite{Noohi}, $\S$17]{\label{D:Serre}}
   We say that a topological stack $\X$ is {\bf Serre} if it is equivalent
   to the quotient stack of a topological groupoid $[s,t \: R\sst{}X]$
   whose source map  (hence, also its target map) is a local Serre
   fibrations. That is, for every $y \in R$, there exists an open
   neighborhood $U\subseteq R$ of $y$  and $V\subseteq X$ of $f(y)$ such
   that the restriction of $s|_U \: U \to V$is a Serre fibration.
 \end{defn}

\begin{lem}{\label{L:square}}
 Let $f \: \X \to \Y$ be a morphism of homotopical stacks
 (resp., paratopological stacks).
 Then, there is a 2-commutative diagram
   $$\xymatrix@=16pt@M=8pt{
           X \ar[r]^g \ar[d]_{\varphi} & Y \ar[d]^{\psi} \\
                              \X \ar[r]_{f} & \Y} $$
 where $X$ and $Y$ are topological spaces (resp. paracompact
 topological spaces) and $\varphi$ and $\psi$
 are universal weak equivalences (resp., parashrinkable morphisms).
\end{lem}

\begin{proof}
 We only prove the case of homotopical stacks.
 By Theorem \ref{T:nice}, we can choose universal weak equivalences
 $\psi\: Y \to \Y$ and  $h \: X \to
 \X\times_{\Y}Y$. (Notice that, by Lemma \ref{L:lim3},
 $\X\times_{\Y}Y$ is homotopical.) Set $\varphi=\pr_1\circ h$ and
 $g=\pr_1\circ h$.
\end{proof}

The following lemmas were proved implicitly in the course of proof of
Lemma \ref{L:adjoint}. We state them separately for future reference.

\begin{lem}{\label{L:12}}
  Let $\varphi \: X \to \X$ be a universal weak equivalence with $X$
  a topological space. Let $f_1,f_2 \: Y \to X$ be continuous maps
  of topological spaces such that $\varphi \circ f_1$ and
  $\varphi\circ f_2 \: Y \to \X$ are 2-isomorphic. Then, $f_1$ and
  $f_2$ are equal in the weak homotopy category $\Top_{w.e.}$ of topological
  spaces.
\end{lem}

\begin{proof}
 Let $g=\varphi\circ f_1$, and consider the  2-cartesian diagram
   $$\xymatrix@=16pt@M=8pt{   Z \ar[r]^h \ar[d]_{\psi}
                              &  X \ar[d]^{\varphi}  \\
                              Y \ar[r]_{g} & \X} $$
 where $Z=Y\times_{\X}X$. The maps $f_1$ and $f_2$ correspond to
 section $s_1,s_2 \: Y \to Z$ of $\psi$. Since $\psi$ is a weak
 equivalence, $s_1$ and $s_2$ are equal in $\Top_{w.e.}$.
 Therefore, $f_1=h\circ s_1$ and $f_2=h\circ s_2$ are also equal in
 $\Top_{w.e.}$.
\end{proof}

\begin{lem}{\label{L:122}}
  Let $\varphi \: X \to \X$ be a parashrinkable  (resp.,
  pseudoshrinkable) morphism (see Definition \ref{D:stackshrinkable})
  with $X$ a paracompact topological space (resp.
  a CW complex). Let $f_1,f_2 \: Y \to X$ be continuous maps of
  topological spaces such that $\varphi \circ f_1$ and $\varphi\circ
  f_2 \: Y \to \X$ are 2-isomorphic. Then, $f_1$ and  $f_2$ are homotopic.
\end{lem}

\begin{proof}
  Copy the proof of Lemma \ref{L:12}.
\end{proof}

\begin{thm}{\label{T:homotopy}}
  Let $(\X,x)$ be a pointed homotopical stack. Then, one can define
  homotopy groups $\pi_n(\X,x)$ that are functorial with
  respect to pointed morphisms of stacks. When $\X$ is a Serre
  topological stack, these homotopy groups are naturally isomorphic
  to the ones defined in (\cite{Noohi}, $\S$17). That is,
  $\pi_n(\X,x)\cong[(S^n,\bullet),(\X,x)]$.
\end{thm}

\begin{proof}
 Let $(\X,x)$ be a pointed homotopical stack. Choose a universal
 weak equivalence $\varphi \: X \to \X$. Pick a point $\tilde{x} \in
 X$ sitting above $x$. (This means, a map $\tilde{x} \: \bullet \to
 X$, together with a 2-morphism $\alpha \: x \Rightarrow
 p\circ\tilde{x}$, which we usually suppress from the notation for
 convenience.) For $n\geq0$, we define $\pi_n(\X,x):=\pi_n(X,\tilde{x})$.

 Let us see why this definition is
 independent of the choice of $\tilde{x}$. Let $\tilde{x}' \in X$ be
 another point above $x$. Let $F=\bullet\times_{\X}X$ be the fiber of
 $\varphi$ over $x$. The map $F \to \bullet$, being the base
 extension of $\varphi$, is   a weak homotopy
 equivalence. This means that $F$ is a
 weakly contractible topological space. The lifts $\tilde{x}$ and
 $\tilde{x}'$ of $x$ correspond to points $\bar{x}$ and $\bar{x}'$ in
 $F$. Since $F$ is weakly contractible, there is a  path $\gamma$,  unique up
 to homotopy, joining $\bar{x}$ and $\bar{x}$. Taking the image of
 $\gamma$ in $X$ we find a natural path joining
 $\tilde{x}$ and $\tilde{x}'$. This path  defines a natural
 isomorphism $\pi_n(X,\tilde{x})\risom \pi_n(X,\tilde{x}')$.

 We will leave it to the reader to verify that $\pi_n(\X,x)$  is also
 independent of the chart $\varphi \: X \to \X$ and that it is functorial.
 The proof makes use of the Lemmas \ref{L:square} and \ref{L:12}.

 In the case where $\X$ is Serre topological, it follows from
 Corollary \ref{C:retractionwe} that the homotopy groups
 defined above are naturally isomorphic to the ones defined in
 (\cite{Noohi}, $\S$17).
\end{proof}

% --------------------------------------------------------------------
\section{(Co)homology theories for homotopical
stacks}{\label{S:Cohomology}}

By virtue of Theorem \ref{T:weakly} we can do algebraic topology on
homotopical stacks by transporting things
back and forth between a stack and its classifying space.
In this section we use this idea to show how we can extend 
generalized (co)homology  theories to stacks.

\medskip \noindent{\bf Fact.} Let $h$ be a (co)homology theory on the
category of topological spaces which
is invariant under weak equivalences. Then $h$ extends naturally to
the category of homotopical stacks. If $h$ is only invariant under homotopy
equivalences,\footnote{Usually, (co)homology theories of
\v{C}ech type, or certain sheaf cohomologies, are only invariant
under homotopy equivalences.}  then $h$ extends naturally to the category of
\hpara\ paratopological stacks (Definition \ref{D:hpara}).

\medskip

Let us show, for example, how to define $h^*(\X,\A)$ for a pair
$(\X,\A)$ of homotopical stacks. (In $\S$\ref{S:Diagram} we will
discuss in detail how to define homotopy types of small diagrams of
stacks.) From now on, we will assume that $h$ is contravariant, and
denote it by $h^*$. Everything we say will be valid for a homology
theory as well.

Pick a universal weak equivalence $\varphi \: X \to \X$ , and set
$A:=\varphi^{-1}\A \subseteq X$. Define $h^*(\X,\A):=h^*(X,A)$. This
definition is independent, up to a natural isomorphism, of the
choice of $\varphi$. To see this, let $\varphi' \: X' \to \X$ be
another universal weak equivalence, and form the fiber product
$\varphi'' \: X'' \to \X$.
    $$\xymatrix@=16pt@M=8pt{  (X'',A'')  \ar[r]^p \ar[d]_{q}
                              &  (X,A) \ar[d]^{\varphi}  \\
                             (X',A') \ar[r]_{\varphi'} & (\X,\A)}$$

Since $p$ and $q$ are weak equivalences of pairs, it follows that
there are natural isomorphisms $h^*(X',A')\cong h^*(X'',A'')\cong
h^*(X,A)$.

Covariance of $h^*$ with respect to morphisms of pairs of
homotopical stacks follows from Lemmas \ref{L:square} and
\ref{L:12}.

For more or less trivial reasons, the resulting  cohomology theory
on the category of homotopical stacks will maintain  all the
reasonable (read functorial) properties/structures   that it has
with topological spaces (e.g., excision, long exact sequence for
pairs, Mayor-Vietoris, products, etc.). Homotopic morphisms (in
particular, 2-isomorphic morphisms) induce the same map on
cohomology groups.

The proof of all of these follows by same line of argument: choose a
universal weak equivalence $\varphi \: X \to \X$, verify the desired
property/structure on $X$, and then use Lemmas \ref{L:square} and
\ref{L:12} to show that the resulting property/structure on
$h^*(\X)$ is independent of the choice of $\varphi$ and is
functorial.

\begin{rem}{\label{R:equivariant}}
 In the case where $h$ is singular (co)homolgoy, what we constructed
 above coincides with the one constructed by Behrend in \cite{Behrend}
 (we will not give the proof of this here).
 When $\X=[X/G]$ is the quotient stack of a topological group action,
 the discussion of $\S$\ref{SS:group} shows the cohomology theories
 defined above coincide with the corresponding
 $G$-equivariant theories constructed via the Borel construction.
\end{rem}

% -------------------------------------------------------
\subsection{(Co)homology theories that are only homotopy invariant}
{\label{SS:homotopyinv}}

There are certain (co)homology theories that are only invariant
under homotopy equivalences. Among these are certain sheaf
cohomology theories or cohomology theories defined via a \v{C}ech
procedure. Such (co)homology theories do not, a priori, extend to
topological stacks because they are not invariant under weak
equivalences. They do, however, extend to hoparacompact 
paratopological stacks $\X$.

The same argument that
we used in the previous subsection applies here, more or less word
by word, as long as we choose the classifying space $X$  of $\X$
to be paracompact and $\varphi\: X \to \X$ to be
parashrinkable. For instance, the reader can easily verify that,
under these assumptions, the morphisms $p$ and $q$ in the
commutative square of the previous subsection will be homotopy
equivalences. This guarantees that $h^*(\X,\A)$ is well-defined. To
prove functoriality one makes use of Lemmas \ref{L:square} and
\ref{L:122}.

\begin{rem}
  The above discussion remains true if we replace the category of
  \hpara\ paratopological stacks by the category
  of {\em \hcw\ pseudotopological stacks} (see Remark
  \ref{R:pseudo}).
\end{rem}

% -------------------------------------
\subsection{A remark on supports}

The notion of {\em supports} for a (co)homology theory can
sometimes be extended to the stack setting. The following result
will not be used elsewhere in the paper, but we include it to
illustrate the idea.

Let us say that a  homology theory $h$ on topological spaces is
(para)compactly supported if for every topological space $X$ the map
    $$\underset{K\to X}{\limdir} h_*(K) \to h_*(X)$$
is an isomorphism. Here, the limit is taken over all maps $K \to X$
with $K$  (para)compact.  
For example, singular homology is  compactly supported.

\begin{prop}{\label{P:support}}
  Let $h$ be a (para)compactly supported homology theory. Then, for every
  paratopological stack $\X$, we have a natural isomorphism
      $$\underset{K\to \X}{\limdir} h_*(K) \risom h_*(\X),$$
  where the limit is taken over all maps $K \to \X$
  with $K$ a (para)compact topological space.
\end{prop}

\begin{proof}
  Choose a parashrinkable morphism $\varphi \: X \to \X$ and use the
  fact that every morphism $K \to \X$ from a paracompact topological
  space $K$ has a lift, unique up to homotopy, to $X$; see Lemma \ref{L:lift}.
\end{proof}

% --------------------------------------------------------------------
\section{Classifying spaces for diagrams of topological
stacks}{\label{S:Diagram}}

In this section, we prove a diagram version of Theorem
\ref{T:weakly}. We show  (Theorem \ref{T:diagramadjoint}) that to every 
small diagram of topological stacks (with a certain condition on the shape 
of the diagram) one can associate a  diagram of classifying topological 
spaces which is well-defined up to (objectwise) weak equivalence
of diagrams.  Theorem \ref{T:diagramadjoint} is actually formulated
in a way that it implies versions of the above statement for
topological, homotopical, paratopological, and pseudotopological
stacks.

 Let $\sfD$ be a
category (which we will think of as a diagram). In what follows, we
will assume  that $\sfC$ and $R$ are any of the following pairs:
\begin{itemize}
  \item[$\mathbf{1}$)] $\mfC$ is $\TopSt$ and $R$ is the class of locally shrinkable
  maps of topological spaces;
  \item[$\mathbf{2}$)]  $\mfC$ is $\HoSt$ and $R$ is the class of universal weak equivalences
   of topological spaces;
  \item[$\mathbf{3}$)] $\mfC$ is $\ParSt$ and $R$ is the class of parashrinkable
  maps of topological spaces;
  \item[$\mathbf{4}$)]  $\mfC$ is $\PsSt$ and $R$ is the class of pseudoshrinkable
  maps of topological spaces.
\end{itemize}

In the first two cases, assume in addition that the category $\sfD$
has the property that for every object $d$ in $\sfD$ there are only
finitely many arrows coming out of $d$. Lemma \ref{L:prod} now
guarantees that in all  four cases at least one of the two
conditions ($\mathbf{A}$) or ($\mathbf{B}$) of $\S$\ref{SS:diagram}
is satisfied.

 By the notation of $\S$\ref{SS:diagram},
$\tilde{R}$ stands for the class of representable morphisms of
stacks which are locally shrinkable, universal weak equivalence,
parashrinkable, or pseudoshrinkable (depending on which pair
$\mathbf{1}$-$\mathbf{4}$ we are considering).

Recall that $\mfC^{\sfD}$ stands for the category of lax functors
$\sfD \to \sfC$. A morphism $\tau$ in $\mfC^{\sfD}$ is a natural
transformation of functors.

\begin{thm}{\label{T:diagramadjoint}}  % WAS P:diagramunivwe
  Let $\sfD$ be a small category, and let $\mfC$ and $R$ be as above.
  Let $T$ (resp., $\tilde{T}$) be the class of all
  transformations $\tau$ in  $\mfC^{\sfD}$ which have
  the property that for every $d \in \sfD$ the corresponding
  morphism $\tau_d$ in $\mfC$ is in $R$ (resp., $\tilde{R}$). Then,
  the  inclusion functor $\Top^{\sfD} \to \mfC^{\sfD}$ induces a
  fully faithful functor $\iota^{\sfD} \: T^{-1}\Top^{\sfD} \to
  T^{-1}[\mfC^{\sfD}]$,  and $\iota^{\sfD}$ has  a right adjoint
  $\Theta^{\sfD}$. Furthermore, $\Theta^{\sfD}$ can be defined so
  that the counits of adjunction are the identity transformations
  and the units of adjunction  are  honest transformations in
  $\tilde{T}$.
\end{thm}

\begin{proof}
  Let $\sfB=\Top$ and use Lemma \ref{L:diagramadjoint}.
\end{proof}

Let $P\: \sfD \to \mfC$ be a diagram of stacks in $\mfC$. The
diagram $\Theta^{\sfD}(P) \: \sfD \to \Top$ should be regraded as
the {\bf weak homotopy type} of $P$. The transformation $\varphi \:
\Theta^{\sfD}(P) \Rightarrow P$ allows one to relate the homotopical
information in the diagram $P$ to the homotopical information in its
homotopy type $\Theta^{\sfD}(P)$. Notice that $\varphi$ is an
objectwise universal weak equivalence.

The following propositions say that the classifying space functor $\Theta \:
R^{-1}\PsSt \to R^{-1}\Top$ of Theorem \ref{T:weakly},  can be
lifted to a functor to $\Top$ if we restrict it to a small sub
2-category  $\mfS$.

\begin{prop}{\label{P:honest1}}
  Let $\mfS$ be a small sub 2-category of the 2-category $\PsSt$ of
  pseudotopological stacks, and denote the inclusion functor by $i_{\mfS}$.
  Identify $\Top$ with a subcategory of
  $\PsSt$. Then, there is a functor $\Theta_{\mfS} \: \mfS \to \Top$
  and a transformation $\varphi_{\mfS} \: \Theta_{\mfS} \Rightarrow
  i_{\mfS}$ such that for every $\X$ in $\mfS$, $\varphi_{\mfS}(\X) \:
  \Theta_{\mfS}(\X) \to \X$ is pseudoshrinkable (in particular, a
  universal weak equivalence). In the case where $\mfS$ sits inside
  $\ParSt$, $\Theta_{\mfS}$ and $\varphi_{\mfS}$ can be chosen so that
  $\varphi_{\mfS}(\X)$ are parashrinkable.
\end{prop}

\begin{proof}
  Follows from Corollary \ref{C:honest}.
\end{proof}

\begin{prop}{\label{P:honest2}}
  Let $\mfS$ be a small sub 2-category of the 2-category $\HoSt$
  of homotopical stacks. Identify $\Top$ with a subcategory
  of $\HoSt$. Assume that $\mfS$ has the property that for every
  stack $\X$ in $\mfS$ there are only finitely many $\Y$ in $\mfS$
  to which there is a morphism from $\X$. Then, there is a functor
  $\Theta_{\mfS} \: \mfS \to \Top$ and a transformation $\varphi_{\mfS} \:
  \Theta_{\mfS} \Rightarrow \id_{\mfS}$ such that for every $\X$ in $\mfS$,
  $\varphi_{\mfS}(\X) \: \Theta_{\mfS}(\X) \to \X$ is an atlas for $\X$ which is a
  universal weak equivalence. In the case where $\mfS$ sits inside
  $\TopSt$, $\Theta_{\mfS}$ and $\varphi_{\mfS}$ can be chosen so that
  $\varphi_{\mfS}(\X)$ are  locally shrinkable.
\end{prop}

\begin{proof}
  Follows from Corollary \ref{C:honest}.
\end{proof}

% ---------------------------------------------------
\subsection{Homotopy types of special diagrams}{\label{SS:special}}

The weak homotopy type of a diagram $\{\X_d\}$ of stacks can be
constructed more easily if we assume that: 1) our diagram category
$\sfD$ has a final object $\star$, 2) the morphisms in the diagram
are representable. For this, choose a locally shrinkable
(parashrinkable, pseudoshrinkable, or a universal equivalence,
depending on which class of stacks we are working with) map
$X_{\star} \to \X_{\star}$, and define $X_d$, $d \in \sfD$, simply
by base extending $X_{\star}$ along the  the morphism $\X_d \to
\X_{\star}$, as in the diagram
      $$\xymatrix@=12pt@M=8pt{  X_d \ar[d] \ar[r] &  X_{\star}\ar[d] \\
           \X_d  \ar[r] &  \X_*  }$$

This construction of the weak homotopy type of a diagram has certain
advantages over the general construction of the previous subsection.
Suppose that every morphism $f$ in $\sfD$ is labeled by a property
$\mathbf{P}_f$ of morphisms of topological spaces  which is
invariant under base change. (Note that such property can then be
extended to representable morphisms of stacks.) Then, it is obvious
that if the morphisms $f$ in a diagram  $\{\X_d\}$ have the
properties $\mathbf{P}_f$, then so will the corresponding morphisms
in the diagram
 $\{X_d\}$.

\begin{ex}
  Let $\sfD=\{1\to2\}$, and assume the label assigned to the unique
  morphism in $\sfD$ is `closed immersion'. Then, it follows that
  every closed pair
  $(\X,\A)$ of  topological stacks has the weak homotopy type of a
  closed pair $(X,A)$ of topological spaces.
  Furthermore, there is a morphism of pairs $\varphi \: (X,A) \to (\X,\A)$
  which is a universal weak equivalence on both terms. This is
  essentially what we discussed in $\S$\ref{S:Cohomology}.

  In the case where $\A$ is a point, $(X,A)$
  will be a pair with $A$  weakly contractible. Therefore, we can define
  $\pi_n(\X,x):=\pi_n(X,A)$. This is exactly what we discussed in
  $\S$\ref{S:Homotopy}.
\end{ex}

% -----------------------------------------------------------
\section{Appendix I: Calculus of right fractions}{\label{A:Calculus}}

Let $\sfC$ be a category.  Let $R$ be a class of morphisms in $\sfC$
which contains all identity morphisms and is closed under
composition and base extension. The localized category $R^{-1}\sfC$
can be calculated using a {\em  calculus of right fractions}, as we
see shortly. Our setting is slightly different from that of  Gabriel-Zisman
 (\oldcite{GaZi}, $\S$2.2) in that their condition (d) may
not be satisfied in our case. We are, however, making a stronger
assumption that $R$ is closed under base extension.

Let $R^{-1}\sfC$ be a category with the same set of objects as
$\sfC$. The morphisms in $R^{-1}\sfC$ are defined as follows. A
morphism  from $X$ to $Y$ is presented by a span $(r,g)$
     $$\xymatrix@=12pt@M=8pt{ & T \ar[ld]_{r} \ar[rd]^g & \\
           X  & & Y  }$$
where $r$ is in $R$ and $g$ is a morphism in $\sfC$. For fixed $X$
and $Y$, the spans between them form a  category $\Roof(X,Y)$. The
morphisms in  $\Roof(X,Y)$ are morphisms  $T' \to T$ in $\sfC$ which
respect the two legs of the spans. By definition, two spans in
$\Roof(X,Y)$ give rise to the same morphism in
$\Hom_{R^{-1}\sfC}(X,Y)$ if and only if they are in the same
connected component of $\Roof(X,Y)$. That is, if they are connected
by a zig-zag of morphisms in $\Roof(X,Y)$. In other words,
  $$\Hom_{R^{-1}\sfC}(X,Y):= \pi_0\Roof(X,Y).$$

The composition of spans is defined in the obvious way. It is easy
to see that $R^{-1}\sfC$ satisfies the universal property of
localization.

\begin{rem}{\label{R:enhance}}
 We can enhance $R^{-1}\sfC$ to a bicategory by defining the
 the hom-category between $X$ and $Y$ to be $\Roof(X,Y)$. The
 localized category
 $R^{-1}\sfC$ is recovered from this bicategory by declaring all
 2-cells to be equalities. That is, by replacing the hom-categories
 $\Roof(X,Y)$ with the set $\pi_0\Roof(X,Y)$.
\end{rem}

Let $f, f' \: X \to Y$ be morphisms in $\sfC$. We say that $f, f'$
are $R$-{\em homotopic} if there is a  commutative diagram
      $$\xymatrix@=45pt@M=4pt{
           X \ar@/^/ [r]^{t} \ar@/^2pc/[rr]_{f}
             \ar@/_/ [r]_{t'} \ar@/_2pc/[rr]^{f'}
            &  V \ar[l] |-{r} \ar[r]^g & Y,   }$$
where $r$ is in $R$. Let $\sim$ be the equivalence relation on
$\Hom_{\sfC}(X,Y)$  generated by $R$-homotopy. We have a natural map
$$\eta \: \Hom_{\sfC}(X,Y)/_{\sim} \to \Hom_{R^{-1}\sfC}(X,Y)$$
 $$f \mapsto (\id,f).$$

\begin{lem}{\label{L:zigzag}}
  Assume that $X \in \sfC$ has the property that every morphism $r \: V
  \to X$ in $R$ admits a section. Then, for every $Y$ in $\sfC$,
  the natural map
   $$\eta \: \Hom_{\sfC}(X,Y)/_{\sim} \to \Hom_{R^{-1}\sfC}(X,Y)$$
  is a bijection.
\end{lem}

\begin{proof}
  Given a span $(r,g)$ from $X$ to $Y$, choose a section $s$ for
  $r$. Then, $g\circ s$, or rather the span $(\id,g\circ s)$,
  represents the same morphism in $R^{-1}\sfC$ as $(r,g)$. This
  shows that $\eta$ is surjective. To prove injectivity, consider
  $f, f' \in \Hom_{\sfC}(X,Y)$. It is easy to see that there is a morphism
  in $\Roof(X,Y)$ between $(\id,f)$ and $(\id,f')$ if and only
  if $f$ and $f'$ are $R$-homotopic. Therefore, the $R$-homotopy
  classes of morphisms in $\Hom_{\sfC}(X,Y)$ correspond precisely to
  the connected components of $\Roof(X,Y)$. This proves injectivity.
\end{proof}

% -----------------------------------------------------------
\section{Appendix II: Relative Kan extensions}{\label{A:Kan}}

We introduce a right Kan extension construction in the setting of
fibered 2-categories $\pi \: \mfU \to \mfB$. In the case where the
base  2-category $\mfB$ is just a point and $\mfU$ is a category,
this reduces to the usual right Kan extension as defined in
(\cite{MacLane}, $\S$X).

Let $\mfB$ and $\mfU$ be 2-categories, and let $\pi \: \mfU \to
\mfB$ be a fibered 2-category (not necessarily in 2-groupoids). It
is sometimes more convenient to think of this fibered 2-category as
the contravariant 2-category-valued lax functor $\mfB \to
2\mathfrak{Cat}$ which assigns to an object $b$ in $\mfB$ the fiber
$\mfU(b)$ of $\mfU$ over $b$. (In our application
($\S$\ref{SS:diagram}), $\pi$ is fibered in 1-categories, so the
corresponding lax functor takes values in $\mathfrak{Cat}$.) For
every morphism $f\: a \to b$ in $\mfB$, we have the {\em pull-back}
functor $f^{\Box} \: \mfU(b) \to \mfU(a)$. (To define $f^{\Box}$ we
need to make some choices,  but the resulting functor $f^{\Box}$
will be unique up to higher coherences.) The laxness of our
2-category-valued functor means that, for every pair of composable
morphisms $f$ and $g$ in $\mfB$, we have a natural transformation
$g^{\Box}\circ f^{\Box} \Rightarrow (f\circ g)^{\Box}$, and that
these transformations satisfy the usual coherence conditions. The
fibered 2-category $\mfU$ can be recovered from this lax functor by
applying the Grothendieck construction.

Let $\pi \: \mfU \to \mfB$ be a fibered 2-category, and  $\Gamma$ a
2-category. Let $b$ be an object in $\mfB$. Consider a diagram $P \:
\Gamma \to \mfU(b)$, that is, a lax functor from $\Gamma$ to
$\mfU(b)$. Assume that $P$ has a limit $\liminv P$ in $\mfU(b)$. Let
$\underline{\liminv P} \: \Gamma \to \mfU(b)$ denote the constant
functor with value $\liminv P$, and let $\Upsilon_P \:
\underline{\liminv P} \Rightarrow P$ be the universal
transformation.

 \begin{defn}{\label{D:limit}}
  The notation being as above, we say that the limit $\liminv P$ of
  $P$ in $\mfU(b)$ is {\bf global}, if for every morphism $f \: a \to
  b$ in $\mfB$,  and  every object $k \in \mfU(a)$, the functor
     $$\Hom_{\mfU,f}(k,\liminv P) \to
          \operatorname{Trans}_f(\underline{k},P)$$
     $$\tilde{f} \mapsto \Upsilon_P\circ\tilde{f}$$
  is an equivalence of categories. Here, $\Hom_{\mfU,f}$ means
  those morphisms in $\mfU$ which map to $f$ under $\pi$. Similarly,
  $\operatorname{Trans}_f$ stands for those transformations $\Phi$
  (of functors $\Gamma \to \mfU$) such that, for every $d \in \sfD$,
  the image of the morphism $\Phi(d)$ under $\pi$ is equal to $f$.
  (Note that both sides are 1-categories. In the case where $\pi \:
  \mfU \to \mfB$ is fibered in 1-categories, they are
  actually equivalent to sets.)
\end{defn}

\begin{rem}{\label{R:relativelimit}}
 The limit  $\liminv P$ being global is equivalent to
  requiring the pullback $f^{\Box}
  (\liminv P) \in \mfU(a)$ to be the limit
  of the pullback diagram $f^{\Box}\circ  P \: \Gamma \to \mfU(a)$,
  for every $f \: a \to b$.
\end{rem}

\begin{defn}{\label{D:relativecomplete}}
  Let $\pi \: \mfU \to \mfB$ be a fibered 2-category, and  $\Gamma$
   a 2-category. Let $b$ be an object  in $\mfB$. We say that $\pi \:
  \mfU \to \mfB$ is $\Gamma$-{\bf complete} at $b$, if every diagram
  $P \: \Gamma \to \mfU(b)$ has a limit and the limit is global
  (Definition \ref{D:limit}). More generally, let $\sfD$ and $\sfE$ be
  2-categories,\footnote{We make an exception to our notational convention
  ($\S$\ref{S:Notation}) that Sans Serif symbols stand for 1-categories,
  because in our application ($\S$\ref{SS:diagram}) $\sfD$ and
  $\sfE$ will be 1-categories.}  and $\sth \: \sfE \to
  \sfD$ and $p \: \sfD \to \mfB$  functors. We say that $\pi \: \mfU
  \to \mfB$ is $\sth$-{\bf complete} at $p$ if it
  is $(d\!\downarrow\!\sfE)$-complete at $p(d)$ for every $d \in \sfD$.
\end{defn}

The comma 2-category $(d\!\downarrow\!\sfE)$ appearing in the above
definition is defined as follows. The objects are pairs
$(e,\alpha)$, where $e \in \Ob\sfE$ and $\alpha \: d \to \sth(e)$ is
a morphism in $\sfD$. A morphism $(e,\alpha)\to (e',\alpha')$ in
$(d\!\downarrow\!\sfE)$ is a morphism $\gamma \: e \to e'$ in
$\sfE$, together with a 2-morphism $\tau \: \sth(\gamma)\circ\alpha
\Rightarrow \alpha'$. A 2-morphism from $(\gamma,\tau)$ to
$(\gamma',\tau')$ is a 2-morphism $\epsilon \: \gamma \Rightarrow
\gamma'$ in $\sfE$ which makes the 2-cell in $\sfD$ consisting of
$\tau$, $\tau'$, and $\sth(\epsilon)$ commute. (Note that in the
case where $\sfD$ and $\sfE$ are 1-categories,
$(d\!\downarrow\!\sfE)$ is also a 1-category and it coincides with
the one defined in  \cite{MacLane}, II.6. If, furthermore, $\sfE$ is
discrete, then $(d\!\downarrow\!\sfE)$ is also discrete.)

\begin{defn}{\label{D:relativprod}}
   Let $\pi \: \mfU \to \mfB$ be a fibered 2-category, and $I$
   an index set. We say that $\pi \: \mfU \to \mfB$ is {\bf
  $I$-complete} (or it {\bf has global $I$-products}) if it is
  $I$-complete at every $b \in \mfB$. Here, we think of $I$ as the
  discrete 2-category with objects $I$ and no nontrivial
  morphisms or 2-morphisms.
\end{defn}

The following lemma shows that completeness with respect to a
functor is invariant under base change of fibered categories.

\begin{lem}{\label{L:baseextension}}
  Let $\pi \: \mfU \to \mfB$ be a fibered 2-category, and
 $\sth \: \sfE \to \sfD$ a functor of 2-categories. Let $\mfB' \to \mfB$
  be a functor, and let $\pi' \: \mfU' \to \mfB'$ be the pullback fibered
  2-category. Let $p' \: \sfD \to \mfB'$ be a functor and $p \: \sfD
  \to \mfB$ the composite functor. Suppose that $\pi \: \mfU \to \mfB$ is
  $\sth$-complete at $p$. Then, $\pi' \: \mfU' \to \mfB'$ is
  $\sth$-complete at $p'$
\end{lem}

\begin{proof}
  Straightforward.
\end{proof}

Let $\sfD$ and $\sfE$ be 2-categories, and fix a ``base'' functor $p
\: \sfD \to \mfB$. Let $\sth \: \sfE \to \sfD$ be a functor and
denote $p\circ \sth$ by $q$.
   $$\xymatrix@=8pt@M=6pt{
     \sfE     \ar[dd]_{\sth}  \ar[rrdd]^q & & \mfU \ar[dd]^{\pi}  \\
                     &  & \\
               \sfD   \ar[rr]_{p}& &  \mfB  }$$
Let $\mfU^{\sfD}_{p}$ be the 2-category of strict lifts of $p$ to
$\mfU$. That is, an object in $\mfU^{\sfD}_{p}$ is a functor $P \:
\sfD \to \mfU$ such that $\pi \circ P=p$. (The latter is an
equality, not a natural transformation of functors.) Define
$\mfU^{\sfE}_{q}$ similarly. Note that in the case where $\pi \:
\mfU \to \mfB$ is fibered in 1-categories $\mfU^{\sfD}_{p}$ and
$\mfU^{\sfE}_{q}$ are 1-categories.

\begin{prop}[Relative right Kan extension]{\label{P:Kan}}
  Notation being as in the previous paragraph, suppose that $\pi \: \mfU \to
  \mfB$ is $\sth$-complete at $p$ (Definition \
  \ref{D:relativecomplete}). Then, the functor $\sth^* \:
  \mfU^{\sfD}_{p} \to \mfU^{\sfE}_{q}$ obtained by precomposing with $\sth$
  admits a right adjoint
      $R\sth \: \mfU^{\sfE}_{q} \to  \mfU^{\sfD}_{p}$.
\end{prop}

\begin{proof}
  Observe that in the case where $\mfB$ is the trivial 2-category
 with one object and $\mfU$ is a 1-category, the proposition
 reduces to the existence of the usual right Kan extension. In fact,
 the construction of  $R\sth$ is simply the imitation of the one
 (\cite{MacLane}, $\S$X). We briefly outline how it is done.

 By base extending $\mfU$ along $p$, we may assume that $\sfD=\mfB$,
 $p=\id$, $q=\sth$. Fix a functor $P \: \sfE \to \mfU$ such that
 $P\circ\pi=\sth$. The desired right Kan extension  $R\sth(P) \: \sfD
 \to \mfU$ of $P$ will then be a section to the projection
 $\pi \: \mfU \to \sfD$.
      $$\xymatrix@R=35pt@C=50pt@M=6pt{  & \mfU \ar[d]^{\pi}  \\
          \sfE  \ar[ru]^P \ar[r]_{\sth} &
           \sfD \ar@/^/@{..>}[u]^(0.35){R\sth(P)}  }$$

 For an object $d \in \sfD$, define a functor $\Psi_d \:
 (d\!\downarrow\!\sfE) \to \mfU(d)$ by the rule
 $\Psi_d(e,\alpha):=\alpha^{\Box}(P(e))$. Observe two things: 1) by
 assumption, $P(e)$ sits above $\sth(e)$, so it makes sense to
 pull it back along $\alpha$; 2)  for every $(e,\alpha)$ there is a
 natural (cartesian) morphism $\eta_{(e,\alpha)} \:
 \alpha^{\Box}(P(e)) \to P(e)$ in $\mfU$ over $\alpha$. (If you want,
 this is the definition of the pullback
 $\alpha^{\Box}(P(e))$.)

  Define the functor $R\sth(P) \: \sfD \to \mfU$ by the rule $d
 \mapsto \liminv\Psi_d$. Note that, by definition, $\liminv\Psi_d$
 is global (Definition \ref{D:limit}).

 Given a morphism $f \: a \to b$ in $\sfD$, the morphism
 $R\sth(P)(f)$ in $\mfU$ is defined as follows. Let $\underline{\liminv\Psi_{a}} \:
 (b\!\downarrow\!\sfE) \to \mfU(a)$ be the constant functor with value
 $\liminv\Psi_{a}$. There is a natural transformation of functors
 $\underline{\liminv\Psi_{a}} \Rightarrow \Psi_b$ over $f$ induced by
 the morphisms
 $\eta_{(e,\alpha)}$ discussed two paragraphs above. Since
 $\liminv\Psi_{b}$ is a global limit, this transformation induces a
 natural morphism $\liminv\Psi_{a} \to \liminv\Psi_{b}$ over $f$. We
 define $R\sth(P)(f)$ to be this morphism. It is readily verified
 that $R\sth(P)$ is a functor (and, obviously,
 $R\sth(P)\circ\pi=\id_{\mfB}$).

 We leave it to the reader to verify that
 $R\sth(P)$ is the desired right Kan extension.
\end{proof}

The right Kan extension $R\sth(P)$ can be illustrated by the
following diagram.
   $$\xymatrix@=14pt@M=6pt{
         \sfE   \ar[rr]^P   \ar[dd]_{\sth}& & \mfU \ar[dd]^{\pi}  \\
                     &  \ar@{=>}[ul]_(0.4){\varepsilon} & \\
   \sfD   \ar[rr]_{p} \ar@{..>}[rruu]_(0.4){R\sth(P)}& &  \mfB  }$$
The lower triangle and the big square in this diagram are strictly
commutative. The natural transformation $\varepsilon$ in the upper
triangle is the counit of adjunction.

\begin{cor}{\label{C:Kan}}
  Let $\pi \: \mfU \to \mfB$ be a fibered 2-category which has
  global products (Definition \ref{D:relativprod}). Let $\sth \: \sfE
  \to \sfD$ be a functor with $\sfD$ a 1-category and $\sfE$ a
  discrete category (i.e., $\sfE$ has no nontrivial morphisms).
  Then, for every functor $p \: \sfD \to \mfB$,  the functor $\sth^*
  \: \mfU^{\sfD}_{p} \to \mfU^{\sfE}_{q}$ obtained by precomposing by
  $\sth$  admits a right adjoint
      $$R\sth \: \mfU^{\sfE}_{q} \to  \mfU^{\sfD}_{p}.$$
  Furthermore, if $\sth$ is so that for every $d \in \sfD$ there
  are only finitely many arrows emanating from $d$ whose target is in
  the image of $\sth$, then the right adjoint $R\sth$ exists
  under the weaker assumption that $\pi \: \mfU \to \mfB$ has global
  finite products.
\end{cor}

\begin{proof}
  It is obvious that $\pi \: \mfU \to \mfB$ is $\sth$-complete for
  every $p \: \sfD \to \mfB$. The result follows from Proposition
  \ref{P:Kan}.
\end{proof}

% ----------------------------------------------------------------
\providecommand{\bysame}{\leavevmode\hbox
to3em{\hrulefill}\thinspace}
\providecommand{\MR}{\relax\ifhmode\unskip\space\fi MR }
% \MRhref is called by the amsart/book/procdefinition of \MR.
\providecommand{\MRhref}[2]{%
  \href{http://www.ams.org/mathscinet-getitem?mr=#1}{#2}
} \providecommand{\href}[2]{#2}

% ----------------------------------------------------------------
% ----------------------------------------------------------------
\end{document}